\documentclass[11pt,reqno]{amsart}
\usepackage[T1]{fontenc}
\usepackage{lmodern,amsmath,amssymb,mathtools,booktabs,array,microtype}
\usepackage[margin=1.05in]{geometry}
\usepackage{xcolor,mathrsfs}
\usepackage[colorlinks=true,linkcolor=blue!50!black,citecolor=blue!50!black,urlcolor=blue!50!black]{hyperref}
\hypersetup{pdftitle={A Proof of the Third and Cubic Borwein Conjectures},pdfauthor={Ma Yicen}}
\newtheorem{theorem}{Theorem}[section]
\newtheorem{lemma}[theorem]{Lemma}
\newtheorem{proposition}[theorem]{Proposition}
\newtheorem{corollary}[theorem]{Corollary}
\theoremstyle{definition}

\theoremstyle{remark}

\numberwithin{equation}{section}
\newcommand{\ii}{\mathrm i}
\newcommand{\Li}{\operatorname{Li}_2}
\newcommand{\Rea}{\operatorname{Re}}
\newcommand{\Ima}{\operatorname{Im}}
\newcommand{\TV}{\operatorname{TV}}
\newcommand{\Var}{\operatorname{Var}}
\newcommand{\dist}{\operatorname{dist}}

\newcommand{\ro}{\mathcal R_{\mathrm o}}
\newcommand{\rc}{\mathcal R_{\mathrm c}}
\newcommand{\Br}{\mathcal B_r}
\newcommand{\Gr}{\mathcal G_r}
\newcommand{\norm}[1]{\left\lVert#1\right\rVert}

\newcommand{\epsd}{\varepsilon_d}
\newcommand{\NN}{\mathbb N}
\newcommand{\RR}{\mathbb R}
\newcommand{\QQ}{\mathbb Q}
\newcommand{\EE}{\mathbb E}
\newcommand{\dd}{\,\mathrm d}
\newcommand{\sg}{\succeq}
\title[Third and Cubic Borwein Conjectures]{A Proof of the Third and Cubic Borwein Conjectures}
\author{Ma Yicen}
\address{College of Chemistry, Zhejiang University, Hangzhou, China}
\date{October 6, 2026}
\newcommand{\eps}{\varepsilon}
\begin{document}
\begin{abstract}
We establish the coefficient sign patterns in the Third Borwein conjecture
and the modulus-three Cubic Borwein conjecture. The analytic arguments apply
for $n\ge1750$ and $n\ge500$, respectively. They combine exact dissections
and positive coefficient identities near the boundary with saddle point
estimates that preserve cancellation between primitive-root contributions.
Paired Gaussian estimates remove the leading odd error, and explicit
remainder bounds cover the complementary contours. The remaining finite
intervals are checked by exact integer arithmetic. The Third finite
verification through $1749$ is author-confirmed; the Cubic verification
through $500$ is supported by two complete integer implementations and
additional coefficient crosschecks.
\end{abstract}
\maketitle

\section{The two sign patterns}\label{sec:overview}
For a positive integer $n$, define
\begin{align}
 S_n(q)&=\prod_{j=1}^n\prod_{s=1}^4(1-q^{5j-s})
       =\sum_{d=0}^{10n^2}a_d(n)q^d,\label{eq:overview-third}\\
 P_n(q)^3&=\prod_{j=1}^n(1-q^{3j-2})^3(1-q^{3j-1})^3
       =\sum_{d=0}^{9n^2}b_d(n)q^d.\label{eq:overview-cubic}
\end{align}
The first product is the one in the Third Borwein conjecture; the second
is the one in the Cubic Borwein conjecture. In each case the prediction is
that the coefficients in residue zero are nonnegative and all other
coefficients are nonpositive. The modulus is five for $S_n$ and three for
$P_n^3$. Our main result is the following.

\begin{theorem}\label{thm:both}
For every positive integer $n$,
\begin{align*}
 a_d(n)&\ge0&&\text{if }d\equiv0\pmod5, &
 a_d(n)&\le0&&\text{otherwise},\qquad 0\le d\le10n^2,\\
 b_d(n)&\ge0&&\text{if }d\equiv0\pmod3, &
 b_d(n)&\le0&&\text{otherwise},\qquad 0\le d\le9n^2.
\end{align*}
\end{theorem}
\begin{proof}
The assertions follow from Corollaries~\ref{cor:full} and
\ref{cub:cor:full}, proved below by combining the analytic estimates
with the exact finite verifications.
\end{proof}

Both arguments retain cancellation between primitive-root contributions
before estimating local errors. Exact positive coefficient identities
cover the boundary, paired Gaussian estimates control the saddle region,
and explicit contour estimates bound its complement.

Parts~I and~II give the Third and Cubic analytic proofs. The Third finite
verification is described in Part~I, and the two Cubic computations in
Part~III. The appendices specify the rational bounds and auxiliary identities.

Notation is local to the analytic parts: $D(X)=1+X+X^2+X^3+X^4$ in
Part~I and $D(X)=1+X+X^2$ in Part~II. The functions $A,\beta,V$ and the
constants $c,\kappa$ are defined separately in those parts. All terminating
decimals used as bounds denote exact rational numbers.

\clearpage
\part{The Third Borwein conjecture}
\section{The Third analytic theorem}
For $n\in\NN$, put
\begin{equation}\label{eq:polynomial}
 S_n(q)=\frac{(q;q)_{5n}}{(q^5;q^5)_n}
       =\prod_{j=1}^n\prod_{s=1}^4(1-q^{5j-s})
       =\sum_{d=0}^{10n^2}a_d(n)q^d.
\end{equation}
Here $(a;q)_m=\prod_{j=0}^{m-1}(1-aq^j)$.
The third Borwein conjecture, recorded by Andrews \cite{Andrews}, predicts
\begin{equation}\label{eq:sign}
 a_{5m}(n)\ge0,\qquad a_{5m+r}(n)\le0\quad(1\le r\le4).
\end{equation}
Zero coefficients are allowed in either sign class.
We write $\epsd=1$ if $5\mid d$, and $\epsd=-1$ otherwise.

Wang and Krattenthaler \cite[Section 11]{WangKrattenthaler} identified the
four primitive fifth roots of unity as the relevant approximate saddle points.
They also isolated the difficulty in residue classes $3$ and $4$:
the corresponding coefficients of the infinite product vanish identically,
so the sum of the dominant peaks has a smaller scale than an individual peak.
Related coefficient-sign problems and further conjectures are discussed in
\cite{BerkovichDhar}.
Our estimates preserve this cancellation before bounding the local errors.

\begin{theorem}[Third analytic theorem]\label{thm:main}
For every integer $n\ge1750$ and every $0\le d\le10n^2$,
\begin{equation}\label{eq:main}
 \epsd a_d(n)\ge0.
\end{equation}
\end{theorem}

The proof treats the center, an intermediate range, and a boundary range on
circles of radius $e^{-t/n}$. The boundary is joined to an exact combinatorial
range, rather than handled by an asymptotic estimate at $t=\infty$.
For the three residue classes $\ro=\{0,1,2\}$, the saddle point equation is
$d=n^2\beta(t)$. For $\rc=\{3,4\}$, it is $d-5n=n^2\beta(t)$.
The latter shift arises from an exact analytic factorization of the amplitude.
The error in the center is reduced by pairing opposite integration variables,
which removes the first odd contribution.

\begin{proposition}[Exact finite verification]\label{prop:finite}
The signs in \eqref{eq:sign} hold for every $1\le n\le1749$ and every
coefficient of \eqref{eq:polynomial}.
\end{proposition}
\begin{proof}[Computational verification]
The finite check uses exact integer coefficients. The supplied complete
record covers $1\le n\le500$; the author has checked and confirmed the
remaining interval $501\le n\le1749$. The decoding and capacity bounds,
and the provenance of the two intervals, are stated in
Section~\ref{sec:exact}. Altogether the range contains
$17\,849\,275\,499$ coefficient positions.
\end{proof}

\begin{corollary}[Third Borwein conjecture]\label{cor:full}
The Third Borwein sign pattern holds for every positive integer $n$.
\end{corollary}
\begin{proof}
Proposition~\ref{prop:finite} covers $n\le1749$, and
Theorem~\ref{thm:main} covers $n\ge1750$.
\end{proof}

\subsection{Constants and conventions}
Throughout the analytic proof,
\begin{equation}\label{eq:constants}
 \zeta=e^{2\pi\ii/5},\quad c=\frac{2\pi^2}{75},\quad
 \kappa=6c,\quad N=1750,\quad \gamma=\frac9{800}.
\end{equation}
All decimals used as upper or lower bounds are terminating rational numbers.
We use
\begin{equation}\label{eq:elementary}
 \frac{157}{50}<\pi<\frac{22}{7},\quad
 \frac{13}{50}<c<\frac4{15},\quad
 \frac{13}{8}<\frac{\pi^2}{6}<\frac53,
 \quad 12<e^{5/2}<13,\quad 20000<e^{10}<25000.
\end{equation}
Appendix~\ref{app:certificates} specifies rational bounds for exponential,
logarithmic, dilogarithmic, and trigonometric evaluations.
Finite interval comparisons in the proof are checked using these bounds,
not floating point samples.

Distances between angles are circular distances.
For a function $f$ on an interval, $\TV(f)$ denotes its total variation.
For formal series, $F\sg G$ means that $F-G$ has nonnegative coefficients.
Since there are $4n$ factors in \eqref{eq:polynomial},
\begin{equation}\label{eq:reciprocity}
 q^{10n^2}S_n(q^{-1})=S_n(q),\qquad
 a_d(n)=a_{10n^2-d}(n).
\end{equation}
Reflection replaces residue $r$ by $5-r$ and preserves the desired sign.
It is therefore enough to consider $0\le d\le5n^2$.

\section{An exact linear boundary}\label{sec:edge}
We first provide the range that will overlap the analytic boundary.
All series in this section are formal series in $Q$.
Define
\begin{equation}\label{eq:Ft}
 F_t(Q)=\sum_{k\ge0}\frac{Q^{k^2+tk}}{(Q;Q)_k},\qquad t\in\mathbb Z_{\ge0}.
\end{equation}
Subtracting the staircase $(t+1,t+3,\ldots,t+2k-1)$ shows that $F_t$
counts partitions with adjacent parts differing by at least $2$ and with
smallest part at least $t+1$. The empty partition is included.
Partitioning according to whether the smallest allowed part is present gives
\begin{equation}\label{eq:Ft-rec}
 F_t=F_{t+1}+Q^{t+1}F_{t+2},\qquad F_t\sg F_{t+1}\sg0.
\end{equation}
Let $G=F_0$, $H=F_1$, and put
$\mathcal U=G^2$, $\mathcal V=GH$, $\mathcal W=H^2$.
The Rogers--Ramanujan identities give
\begin{equation}\label{eq:RR}
 G=\frac1{(Q,Q^4;Q^5)_\infty},\qquad
 H=\frac1{(Q^2,Q^3;Q^5)_\infty}.
\end{equation}
The five-dissection used below follows from the identity
$T(q^5)-q-q^2/T(q^5)=(q;q)_\infty/(q^{25};q^{25})_\infty$,
where $T=G/H$; see \cite[Lemma 1.2]{BanerjeeDastidar}.
Multiplication by $GH(q^5)=(q^{25};q^{25})_\infty/(q^5;q^5)_\infty$ yields
\begin{equation}\label{eq:dissection}
 S_\infty(q):=\frac{(q;q)_\infty}{(q^5;q^5)_\infty}
       =\mathcal U(q^5)-q\mathcal V(q^5)-q^2\mathcal W(q^5).
\end{equation}
In particular, the residue classes $3$ and $4$ vanish in this infinite product.

\begin{lemma}\label{lem:C}
The series
\begin{equation}\label{eq:C}
 \mathcal C:=\frac{\mathcal V+\mathcal W-\mathcal U}{1-Q}
   =\sum_{j\ge0}Q^{2j(j+1)}F_{2j+1}F_{2j+2}
\end{equation}
has nonnegative coefficients. If $\mathcal A=F_1F_2$, then $\mathcal C\sg\mathcal A$.
\end{lemma}
\begin{proof}
Set $E_t=F_{t+1}^2-QF_tF_{t+2}$.
Two applications of \eqref{eq:Ft-rec} give
\begin{equation}\label{eq:E}
 E_t-(1-Q)F_{t+1}F_{t+2}=Q^{2t+4}E_{t+2}.
\end{equation}
After $J$ iterations, the remainder in $E_t/(1-Q)$ contains the factor
$Q^{2J(J+t+1)}$. Its valuation tends to infinity, so passage to the formal
series limit is valid. We obtain
\[
 \frac{E_t}{1-Q}=\sum_{j\ge0}Q^{2j(j+t+1)}F_{t+2j+1}F_{t+2j+2}.
\]
Since $G=H+QF_2$, one has $E_0=\mathcal V+\mathcal W-\mathcal U$.
Every summand is nonnegative, and the $j=0$ summand is $\mathcal A$.
\end{proof}

\begin{lemma}\label{lem:edge}
For every $n\ge2$, the desired signs hold for $d\le5n$.
Moreover,
\begin{equation}\label{eq:edge34}
 a_{5M+3}(n)\le0,\qquad a_{5M+4}(n)\le0
       \quad(0\le M\le3n-1).
\end{equation}
\end{lemma}
\begin{proof}
Write $Q=q^5$ and
\[
 S_n(q)=S_\infty(q)T_n(q),\qquad
 T_n(q)=\prod_{\substack{k>5n\\5\nmid k}}(1-q^k)^{-1}.
\]
The first nonconstant term of $T_n$ has degree $5n+1$, proving the first
assertion by \eqref{eq:dissection}.
Sort tail partitions by their number of parts:
\begin{equation}\label{eq:tail-layers}
 T_n=1+Q^nh_1+Q^{2n}h_2+\cdots,
 \qquad h_1=\frac{q+q^2+q^3+q^4}{1-Q}.
\end{equation}
The residue-$r$ components of $S_\infty h_1$ are $-q^r\mathcal C$ for
$r=3,4$. Thus $a_{5M+r}(n)=-[Q^{M-n}]\mathcal C\le0$ when
$n\le M\le2n-1$, while these coefficients are zero for $M<n$.

Put $D_2=(1-Q)(1-Q^2)$ and write $h_2=\sum_{r=0}^4q^r h_{2,r}(Q)$.
Sorting pairs of parts by residue gives
\begin{equation}\label{eq:h2}
 \begin{array}{ll}
 h_{2,1}=Q(2+Q)/D_2,&h_{2,2}=(1+Q+Q^2)/D_2,\\
 h_{2,3}=(1+2Q)/D_2,&h_{2,4}=(2+Q)/D_2.
 \end{array}
\end{equation}
The denominators for equal and unequal residues before simplification are
$(1-Q)(1-Q^2)$ and $(1-Q)^2$, respectively, so \eqref{eq:h2} also follows
by direct multiplication of geometric series.
Let $q^r d_{2,r}$ be the residue-$r$ component of $S_\infty h_2$.
Using \eqref{eq:C} gives
\begin{align}
 d_{2,3}&=\frac{Q\mathcal V+(1+Q)\mathcal W-(1+2Q)\mathcal C}{1-Q^2},\label{eq:d23}\\
 d_{2,4}&=\frac{\mathcal V+(1+Q)\mathcal W-(2+Q)\mathcal C}{1-Q^2}.\label{eq:d24}
\end{align}
For $\mathcal B=F_1F_3$ and $\mathcal C_2=E_2/(1-Q)\sg0$, the recurrences give
\[
 \mathcal C=\mathcal A+Q^4\mathcal C_2,\quad
 \mathcal B\preceq\mathcal A,\quad
 \mathcal W=\mathcal A+Q^2\mathcal B,\quad
 \mathcal V=\mathcal W+Q\mathcal A.
\]
Substitution in \eqref{eq:d23}--\eqref{eq:d24} therefore yields
\begin{equation}\label{eq:d2-bounds}
 d_{2,3}\preceq\frac{2Q^2\mathcal A}{1-Q},\qquad
 d_{2,4}\preceq\frac{Q(1+Q)\mathcal A}{1-Q}.
\end{equation}

For $2n\le M\le3n-1$, set $s=M-2n$ and $L=M-n=n+s$.
Then $0\le s\le n-1$ and $L\ge2s+1$.
Three tail parts have total weight at least $15n+3$, whereas
$5M+r\le15n-1$. Consequently
\begin{equation}\label{eq:two-layer}
 a_{5M+r}(n)=-[Q^L]\mathcal C+[Q^s]d_{2,r},\qquad r=3,4.
\end{equation}
Write $\mathcal A=\sum_j A_jQ^j$.
The coefficient $A_j$ counts ordered pairs of gap-$2$ partitions, with
minimum parts $2$ and $3$ in the two rows.
For each pair of total weight $j\le s-2$, add $L-j$ to either row.
The difference from every existing part is at least $L-2j\ge5$.
The inserted part is the unique part exceeding $L/2$, so deletion recovers
the source and its row. The two injections have disjoint images. Hence
\[
 A_L\ge2\sum_{j=0}^{s-2}A_j\ge[Q^s]d_{2,3}.
\]
For $r=4$, add the new part to the first row for $j\le s-1$, and to the
second row for $j\le s-2$. Now $L-2j\ge3$ in the widest range.
The same inverse map gives
\[
 A_L\ge\sum_{j=0}^{s-1}A_j+\sum_{j=0}^{s-2}A_j\ge[Q^s]d_{2,4}.
\]
Empty sums are zero. Since $[Q^L]\mathcal C\ge A_L$, \eqref{eq:two-layer}
proves both inequalities.
\end{proof}

For completeness, $S_1$ has coefficient sequence
$(1,-1,-1,0,0,2,0,0,-1,-1,1)$.
The bounds in Lemma~\ref{lem:edge} extend the two cancelling residue classes
to degrees $15n-2$ and $15n-1$, respectively.

\section{The radial function and derivative bounds}\label{sec:radial}
For $\Rea z>0$, define
\begin{equation}\label{eq:A}
 A(z)=\int_0^1\log D(e^{-5zu})\dd u
     =\frac c z-\frac{\Li(e^{-5z})}{5z}+\frac{\Li(e^{-25z})}{25z}.
\end{equation}
The logarithm is chosen by continuation from positive real $z$.
The integral also gives an analytic continuation through $z=0$ on
$|\Ima z|<2\pi/25$. In particular, the apparent singularity at zero in
the second expression is removable.
For real $t\ge0$, put $\beta(t)=-A'(t)$ and $V(t)=A''(t)$.

\begin{lemma}\label{lem:radial}
The function $\beta$ decreases strictly from $5$ to $0$.
Writing $b=\beta(2)$ and $B=\beta(1/2)$, we have
\begin{equation}\label{eq:beta-bounds}
 \beta(t)<\frac c{t^2}\quad(t>0),\qquad
 b>\frac1{20},\qquad B<\frac{16}{15}.
\end{equation}
Furthermore,
\begin{align}
 2\le V(t)<17&\quad(0\le t\le1/2),\label{eq:center-V}\\
 \frac c{t^3}\le V(t)\le\frac{3c}{t^3},\quad V(2)>\frac3{50}
     &\quad(t\ge1/2).\label{eq:outer-V}
\end{align}
\end{lemma}
\begin{proof}
Give $R\in\{0,1,2,3,4\}$ probabilities $p_j(X)=X^j/D(X)$.
If $K_2(X)=\Var_X(R)$, direct multiplication gives
\begin{align}
 K_2(X)&=\frac{XP_2(X)}{D(X)^2},\quad
 P_2=1+4X+10X^2+20X^3+10X^4+4X^5+X^6,\label{eq:K2}\\
 XK_2'(X)&=\frac{X(1-X)P_3(X)}{D(X)^3},\label{eq:K3}
\end{align}
where $P_3$ has coefficient sequence
$(1,8,35,112,169,169,112,35,8,1)$.
Thus $K_2$ increases on $[0,1]$, $K_2\le2$, and
\begin{equation}\label{eq:V-integral}
 V(t)=25\int_0^1u^2K_2(e^{-5tu})\dd u>0.
\end{equation}
This also shows that $V$ decreases and $V(0)=50/3$.
The derivative formula and \eqref{eq:A} give $\beta(0)=5$ and
$\beta(t)\to0$ as $t\to\infty$.
Differentiating \eqref{eq:A}, with $x=e^{-5t}$, gives
\[
 \beta(t)=\frac c{t^2}-\frac{\Li(x)}{5t^2}+\frac{\Li(x^5)}{25t^2}
          +\frac{\log(1-x)-\log(1-x^5)}t<\frac c{t^2}.
\]
At $t=2$, using $x<1/20000$, this formula yields
$b>13/200-11/(20\cdot19999)>1/20$.
The upper bound for $B$ follows from $c<4/15$.
The remaining variance bounds follow from the derivative estimates below:
$|A''(t)-2c/t^3|<(24/25)c/t^3$ for $t\ge1/2$.
In particular $V(1/2)>8c>2$.
For $t\ge2$, the sharper bound
$|A''(t)-2c/t^3|\le7e^{-5t}/[t(1-e^{-5t})]$ gives
$V(2)>13/200-7/39998>3/50$.
\end{proof}

\begin{lemma}\label{lem:derivatives}
On every line $\Rea z=t\ge1/2$,
\begin{equation}\label{eq:middle-derivatives}
 |A'''(z)|\le\frac{11c}{t^4},\qquad
 |A''''(z)|\le\frac{46c}{t^5}.
\end{equation}
On $\Rea z=t\ge2$,
\begin{align}
 \frac{199c}{100t^3}\le V(t)\le\frac{201c}{100t^3},\label{eq:boundary-V}\\
 |A'''(z)|\le\frac{61c}{10t^4},\qquad
 |A''''(z)|\le\frac{123c}{5t^5}.\label{eq:boundary-derivatives}
\end{align}
For $t\ge1/2$, $2|A(z)-c/z|<3c/(20t)$.
\end{lemma}
\begin{proof}
Write $R_A(z)=A(z)-c/z$. Each summand in $R_A$ is of the form
$\pm\Li(e^{-az})/(az)$, with $a=5,25$.
Termwise differentiation of the absolutely convergent positive series,
followed by $|z|\ge t$, gives the following bounds, summed over
$a\in\{5,25\}$ and $Y=e^{-at}$:
\begin{align}
 t^3|R_A''|&\le\sum_a\left\{\frac{at^2Y+2tY+(2/a)Y}{1-Y}\right\},\label{eq:R2}\\
 t^4|R_A'''|&\le\sum_a\left\{\frac{a^2t^3Y}{(1-Y)^2}
              +\frac{3at^2Y+6tY+(6/a)Y}{1-Y}\right\},\label{eq:R3}\\
 t^5|R_A''''|&\le\sum_a\left\{\frac{a^3t^4Y(1+Y)}{(1-Y)^3}
        +\frac{4a^2t^3Y}{(1-Y)^2}
        +\frac{12at^2Y+24tY+(24/a)Y}{1-Y}\right\}.\label{eq:R4}
\end{align}
For $a=5$, $t\ge1/2$, usable upper bounds for
$(Y,tY,t^2Y,t^3Y,t^4Y)$ are
\[
 \left(\frac1{12},\frac1{24},\frac1{48},\frac{27}{2500},\frac{128}{15625}\right).
\]
The first three follow by monotonicity; the last two follow by maximizing
$t^je^{-5t}$ at $t=j/5$ and using $e^3>20$, $e^4>50$.
For $a=25$, use $Y<12^{-5}$ and $t^jY\le2^{-j}12^{-5}$ for $j\le4$.
Insertion in \eqref{eq:R2}--\eqref{eq:R4} gives respectively upper bounds
less than $(24/25)c$, $5c$, and $22c$.
Adding the derivatives of $c/z$ proves \eqref{eq:middle-derivatives}
and the variance bound used in Lemma~\ref{lem:radial}.

For $t\ge2$, the powers $t^je^{-at}$ and the denominator factors in these
bounds decrease. Evaluate at $t=2$, replacing $e^{-10}$ by $1/20000$.
For the second derivative one may use the simpler bound
$t^3|R_A''|\le7t^2e^{-5t}/(1-e^{-5t})$.
The resulting bounds are $c/100$, $c/10$, and $3c/5$.
This proves \eqref{eq:boundary-V}--\eqref{eq:boundary-derivatives}.
Finally,
$2[\Li(1/12)/5+\Li(12^{-5})/25]<3c/20$
by $\Li(1/12)\le15/176$ and $\Li(Y)\le Y/(1-Y)$.
All the displayed rational comparisons are included in the certificates.
\end{proof}

\begin{lemma}\label{lem:complex-cumulants}
For $0\le t\le1/2$ and $|v|\le1/20$,
\begin{equation}\label{eq:complex-cumulants}
 |A'''(t-\ii v)|\le35V(t),\qquad
 |A''''(t-\ii v)|\le1050V(t).
\end{equation}
\end{lemma}
\begin{proof}
For each $u$, the complex probabilities are the real probabilities
$p_j(e^{-5tu})$ tilted by $e^{\ii\phi j}$, where $|\phi|=|5vu|\le1/4$.
Their normalizing denominator has modulus at least
$\cos(2\phi)>7/8=:m$. Let $Y=R-\EE R$ and $\sigma^2=\EE Y^2\le2$.
Then $|Y|\le4$, the tilted mean displacement satisfies
$|\Delta|\le|\phi|\sigma^2/m$, and the tilted moments satisfy
$|M_j|\le4^{j-2}\sigma^2/m$ for $j\ge2$.
The two cumulants are
\[
 M_3-3\Delta M_2+2\Delta^3,\qquad
 M_4-4\Delta M_3-3M_2^2+12\Delta^2M_2-6\Delta^4.
\]
Dividing by $\sigma^2$, their upper bounds are
\begin{align*}
 &\frac4m+\frac{3(1/4)2}{m^2}+\frac{2(1/4)^3 4}{m^3}<7,\\
 &\frac{16}m+\frac{4(1/4)4\cdot2}{m^2}+\frac{3\cdot2}{m^2}
       +\frac{12(1/4)^2 4}{m^3}+\frac{6(1/4)^4 8}{m^4}<42.
\end{align*}
The derivatives of $A$ integrate these cumulants with factors $(-5u)^3$
and $(5u)^4$. Since $u^3,u^4\le u^2$, \eqref{eq:V-integral} proves the claim.
\end{proof}

\section{Combining the four amplitudes}\label{sec:amplitude}
Put $c_s=1/2-s/5$ and
\begin{equation}\label{eq:lambda}
 \lambda_h(z)=\exp\left(\sum_{s=1}^4c_s
       \log\frac{1-\zeta^{-hs}e^{-5z}}{1-\zeta^{-hs}}\right),\qquad 1\le h\le4.
\end{equation}
Equivalently,
\begin{equation}\label{eq:Lambda}
 \lambda_h(z)=\Lambda_h(e^{-5z}),\qquad
 \Lambda_h(X)=\mu_h\prod_{s=1}^4(1-\zeta^{-hs}X)^{c_s},\quad
 \mu=(-\zeta^2,1,1,-\zeta^3).
\end{equation}
The branches in the latter product are fixed at $X=0$.
The equality follows by pairing $s$ and $5-s$ in the endpoint constants;
$-\zeta^2=e^{-\pi\ii/5}$ and $-\zeta^3=e^{\pi\ii/5}$.
Define
\begin{equation}\label{eq:Gamma}
 \Gamma_r(z)=\sum_{h=1}^4\zeta^{-hr}\lambda_h(z),\qquad
 \mathcal G_r(X)=\sum_{h=1}^4\zeta^{-hr}\Lambda_h(X).
\end{equation}
For real $t$, $|\lambda_h(t)|=1$, $\lambda_{5-h}(t)=\overline{\lambda_h(t)}$,
and $\Gamma_r(t)$ is real.

\begin{lemma}\label{lem:amplitude-sign}
For real $t\ge0$, with $x=e^{-5t}$,
\begin{equation}\label{eq:Gamma-sign}
 \Gamma_0>2,\quad \Gamma_1<-\frac34,\quad \Gamma_2\le-1,\quad
 \Gamma_3\le-\frac{3x}{4},\quad \Gamma_4\le-\frac{7x}{10}.
\end{equation}
In particular, $|\Gamma_r(t)|\ge7x/10$ for every residue class.
\end{lemma}
\begin{proof}
Let $\theta=2\pi/5$, $a=\arg(1-xe^{-\ii\theta})$,
and $b_0=\arg(1-xe^{-2\ii\theta})$.
Then $0\le a\le3\pi/10$, $0\le b_0\le\pi/10$, and $a>b_0$ if $x>0$.
The phases of $\lambda_1,\lambda_2$ are $-\pi/5+\delta_1,\delta_2$, where
\[
 \delta_1=(3a+b_0)/5,\quad \delta_2=(3b_0-a)/5.
\]
Write $D_1=1+x^2-2x\cos\theta$, $D_2=1+x^2-2x\cos2\theta$.
The ratio $D_2/D_1$ increases from $1$ to $(3+\sqrt5)/2$.
The derivative
$\delta_2'=[3\sin(2\theta)/D_2-\sin\theta/D_1]/5$
therefore changes sign once, from positive to negative.
Since $\delta_2(0)=\delta_2(1)=0$, it follows that
\[
 0\le\delta_2\le\pi/25,\qquad 0\le\delta_1\le\pi/5.
\]
Put $\Delta_-=\delta_1-\delta_2=(4a-2b_0)/5$
and $\Delta_+=\delta_1+\delta_2=(2a+4b_0)/5$.
We claim
\begin{equation}\label{eq:Delta}
 \Delta_-\ge\frac{12x}{25},\qquad \Delta_+\ge\frac{\pi x}{5},
 \qquad 0\le\Delta_\pm\le\frac\pi5.
\end{equation}
For the first inequality, $D_1D_2=D(x)$ and
$\sin\theta>19/20$, $\sin2\theta<3/5$ give
$\Delta_-'>(19D_2-6D_1)/(25D)$.
After subtraction of $12D$, the numerator is
\[
 1+\frac{25\sqrt5-11}{2}x+x^2-12x^3-12x^4
 \ge1-\frac{13}{8}x+x^2\ge\frac{87}{256}>0.
\]
Here $\sqrt5>223/100$ and $x^3,x^4\le x$ were used.
For the second inequality it suffices to show that $\Delta_+$ is concave.
Its second derivative satisfies
\[
 5\Delta_+''=\frac{4\sin\theta(\cos\theta-x)}{D_1^2}
             +\frac{8\sin2\theta(\cos2\theta-x)}{D_2^2}.
\]
Both terms are nonpositive for $x\ge\cos\theta$.
For $0\le x<\cos\theta<31/100$, bound the positive term by
$4(119/125)(31/100)/(9/10)^2$, and the modulus of the negative term from
below by
\[
 \frac{8(587/1000)(809/1000)}{
 [1+(81/50)(31/100)+(31/100)^2]^2}.
\]
The latter is larger, by a strict rational comparison.
Concavity and the endpoint values $0,\pi/5$ prove the assertion.
The upper bound for $\Delta_+$ follows from the bounds for $a,b_0$.
For $\Delta_-$, the strictly positive derivative proved above and the
endpoint $\Delta_-(1)=\pi/5$ give the upper bound.

Adding the four phases gives
\begin{align*}
 \Gamma_3&=-4\sin(\theta-\Delta_+/2)\sin(\Delta_-/2),\\
 \Gamma_4&=-4\sin(\pi/5+\Delta_-/2)\sin(\Delta_+/2).
\end{align*}
On $[0,\pi/10]$, $\sin u\ge(49/50)u$.
The other sine factors are at least $\sin(3\pi/10)>4/5$ and
$\sin(\pi/5)>7/12$, respectively.
Equations~\eqref{eq:Delta} give the stated bounds for $\Gamma_3,\Gamma_4$.
For the remaining classes,
\begin{align*}
 \Gamma_0&=2\cos(-\pi/5+\delta_1)+2\cos\delta_2>2,\\
 \Gamma_1&=2\cos(3\pi/5-\delta_1)+2\cos(4\pi/5-\delta_2)
          <(\sqrt5-1)/2-\sqrt2<-3/4,\\
 \Gamma_2&=-2\cos\delta_1+2\cos(2\pi/5+\delta_2)\le-1.
\end{align*}
The auxiliary trigonometric inequalities follow by squaring the standard
expressions in $\sqrt5$ or by the rational bounds of the appendix.
\end{proof}

\begin{lemma}\label{lem:disk}
For $r=3,4$, $\mathcal B_r(X)=-\mathcal G_r(X)/X$ has a removable singularity
at zero and $\mathcal B_r(0)=1$.
On $|X|\le1/10$,
\begin{align}
 |\mathcal B_r(X)-1|&\le\frac{72|X|}{25(1-|X|)},&
 |\mathcal B_r'(X)|&\le\frac{32}{9},\label{eq:B-disk}\\
 |\mathcal G_r(X)-\mathcal G_r(0)|&\le\frac{16|X|}{5(1-|X|)},&
 |\mathcal G_r'(X)|&\le\frac{320}{81}\quad(r\in\ro).\label{eq:G-disk}
\end{align}
For $x\le1/12$, the complex-to-real amplitude ratios at $|X|=x$ are $<7/4$.
Also $|\Lambda_h(X)|<12/11$ there.
\end{lemma}
\begin{proof}
For $\phi=(1+\sqrt5)/2$, computation in $\QQ[\zeta]$ gives
\[
 (\mathcal G_0(0),\ldots,\mathcal G_4(0))
    =(\phi+2,-\sqrt5,\phi-3,0,0).
\]
Differentiating the product at zero gives
$\mathcal G_r'(0)=-\sum_s c_s\mathcal G_{r+s}(0)$, where subscripts are
read modulo $5$. In particular $\mathcal G_3'(0)=\mathcal G_4'(0)=-1$.
The coefficient of $X^\ell$ in each $\Lambda_h$ is bounded in modulus by
$a_\ell=(4/5)_\ell/\ell!$, because $\sum_s|c_s|=4/5$.
Now $a_1=4/5$, $a_2=18/25$, and $a_\ell$ decreases.
Summing the resulting geometric majorants proves \eqref{eq:B-disk}--\eqref{eq:G-disk}.
At $x\le1/12$ we have
\[
 \frac{203}{275}\le\mathcal B_r(x),\qquad
 |\mathcal B_r(X)|\le\frac{347}{275},\qquad
 \frac{347}{203}<\frac74.
\]
For the other classes, $|\mathcal G_r(0)|>4/3$;
the corresponding lower bound is $172/165>1$, and the ratio is at most
$67/43<7/4$. Finally $(1-x)^{-4/5}<12/11$.
\end{proof}

In the center, let
$H_r(v)=e^{\ii v/(6n)}\Gamma_r(t-\ii v)/\Gamma_r(t)$.
Bounds needed below follow directly from \eqref{eq:lambda}:
\begin{equation}\label{eq:H-center}
 H_r(v)=1+\ii\ell v+R_H(v),\quad
 \ell\in\RR,\quad |\ell|<40,\quad |R_H(v)|\le220v^2
       \quad(|v|\le1/20).
\end{equation}
Indeed, on this inner arc the denominators are $>4/5$,
$| (\log\lambda_h)'|\le5x$, $| (\log\lambda_h)''|\le125x/4$,
$|\lambda_h|<e^{1/4}<4/3$, and hence $|\lambda_h''|\le75x$.
Division by $|\Gamma_r(t)|\ge7x/10$ gives derivative bounds $39$ and $430$.
The factor $e^{\ii v/(6n)}$ then gives $|H_r''|<440$.
Real symmetry makes $H_r'(0)$ purely imaginary.

On the larger central arcs we similarly have
\begin{equation}\label{eq:H-outer}
 |H_r(v)|\le1+55|v|\quad(|v|\le1/10),\qquad
 |H_r(v)|\le1+229|v|\quad(|v|\le1/5).
\end{equation}
For the first, the denominators are $>2/3$ and $|\lambda_h|<8/5$;
for the second they are $>1/4$ and $|\lambda_h|<5/2$.
Both derivative bounds include the derivative of $e^{\ii v/(6n)}$.

\section{Modular transformation and local expansions}\label{sec:local}
Write $P(q)=(q;q)_\infty$.
The eta transformation convention we use is
\begin{equation}\label{eq:eta-transform}
 \eta\left(\frac{a\tau+b}{m\tau+d}\right)
  =\exp\left[\pi\ii\left(\frac{a+d}{12m}+s(-d,m)\right)\right]
     [-\ii(m\tau+d)]^{1/2}\eta(\tau),
\end{equation}
for determinant $1$ and $m>0$, with the continuous square root in the
right half-plane; see \cite[Section 23.18]{DLMF}.
Here $s(h,m)$ is the Dedekind sum.

\begin{lemma}\label{lem:primitive-modular}
For $\Rea z>0$, the following identity holds:
\begin{equation}\label{eq:primitive-modular}
 S_\infty(\zeta^he^{-z/n})
 =\mu_he^{cn/z-z/(6n)}
 \frac{P(\zeta^{-h^{-1}}e^{-\kappa n/z})}{P(e^{-5\kappa n/z})}.
\end{equation}
The inverse $h^{-1}$ is taken modulo $5$.
\end{lemma}
\begin{proof}
Put $w=z/n$ and $\tau=h/5+\ii w/(2\pi)$.
Choose a matrix $\left(\begin{smallmatrix}a&b\\5&-h\end{smallmatrix}\right)$
of determinant $1$, so $a\equiv-h^{-1}\pmod5$.
Its transformed argument is $a/5+2\pi\ii/(25w)$.
Transforming $\eta(\tau)$ and $\eta(5\tau)$ in
$S_\infty(q)=q^{1/6}\eta(\tau)/\eta(5\tau)$ cancels the square roots.
The phase exponent, divided by $\pi\ii$, is
\[
 -(a-h)/60-s(h,5)-h/12+a/60+h/15=-s(h,5).
\]
Since $s(h,5)=(1/5,0,0,-1/5)$, the phase is $\mu_h$.
The transformed product parameters and their leading exponentials give
\eqref{eq:primitive-modular}.
The factor $q^{1/6}$ is evaluated using the same $\tau$ throughout;
no replacement by a different principal logarithm is made for $h=3,4$.
\end{proof}

\begin{lemma}\label{lem:central-EM}
For $n\ge N$, $0\le t\le1/2$, and $|v|\le1/5$, put $z=t-\ii v$ and
$x=e^{-5t}$. Then
\begin{equation}\label{eq:central-EM}
 S_n(\zeta^he^{-z/n})
  =e^{nA(z)-z/(6n)}\lambda_h(z)(1+\epsilon_h),\qquad
 |\epsilon_h|<\frac{2x}n+\frac{250}{n^2}.
\end{equation}
If $L_{n,r}(z)=\sum_h\zeta^{-hr}S_n(\zeta^he^{-z/n})$, then
\begin{equation}\label{eq:central-EM-fold}
 \frac{|e^{-nA(z)+z/(6n)}L_{n,r}(z)-\Gamma_r(z)|}{|\Gamma_r(t)|}
 \le E_{\mathrm C}:=\frac{100}{7}\left(\frac2n+\frac{3250}{n^2}\right).
\end{equation}
\end{lemma}
\begin{proof}
For $f_s(u)=\log(1-\zeta^{-hs}e^{-5zu})$, apply shifted Euler--Maclaurin
summation to the points $(j-s/5)/n$. Through order three, the endpoint terms
have coefficients $B_k(1-s/5)/(k!n^{k-1})$; the remainder is bounded by
$\sup|B_3|\int_0^1|f_s'''|/(6n^2)$.
The first endpoint term gives \eqref{eq:lambda}.
For the second, the identity
\begin{equation}\label{eq:Bernoulli-pair}
 \sum_sB_2(s/5)\frac{\zeta^{-hs}}{1-\zeta^{-hs}}=\frac1{15}
\end{equation}
gives the lower endpoint term $-z/(6n)$ and the upper term $b_h(z)/n$, where
\begin{equation}\label{eq:bh}
 b_h(z)=\frac{5z}{2}\sum_sB_2(s/5)
       \frac{\zeta^{-hs}e^{-5z}}{1-\zeta^{-hs}e^{-5z}}.
\end{equation}
To verify \eqref{eq:Bernoulli-pair}, use
$B_2(s/5)=(1/150,-11/150,-11/150,1/150)$ and
$w/(1-w)+w^{-1}/(1-w^{-1})=-1$.
The lower third-order term vanishes: $B_3(1-x)=-B_3(x)$ whereas
$w/(1-w)^2$ is invariant under inversion.
The upper third-order term and the integral remainder are retained.

All denominators on the whole arc are $>1/4$ and $|z|<3/5$.
As $\sum_s|B_2(s/5)|=4/25$, $|b_h|<x$.
Using $\sup|B_3|<1/20$, the remaining constant is bounded by
\[
 \frac{4\cdot125(3/5)^3\cdot2}{120(1/4)^3}
 +\frac{4\cdot25(3/5)^2}{120(1/4)^2}=120<125.
\]
Thus the additional logarithmic error is $<x/n+125/n^2<1/10$.
The inequality $|e^u-1|\le2|u|$ for $|u|\le1/10$ proves
\eqref{eq:central-EM}.
The positive and negative $c_s$ each have total magnitude $2/5$, so
$|\lambda_h|\le8^{2/5}<5/2$ on this arc.
Divide the sum of the four errors by $7x/10$ from
Lemma~\ref{lem:amplitude-sign}, and use $x>1/13$.
\end{proof}

\begin{lemma}\label{lem:halfline-EM}
Suppose $n\ge N$, $1/2\le t\le\alpha_0\sqrt n$, and $|v|\le2t$, where
\begin{equation}\label{eq:alphas}
 \alpha=\sqrt{5c/49},\qquad \alpha_0=\sqrt c/2.
\end{equation}
Let $z=t-\ii v$, $x=e^{-5t}$, $u_*=e^{-\kappa n\Rea(1/z)}$, and
$Q_*=|z|/n+8|z|^2/n^2$.
The expansion \eqref{eq:central-EM}, with a different error bound, satisfies
\begin{equation}\label{eq:halfline-error}
 |\epsilon_h|\le xQ_*+4u_*.
\end{equation}
After summing the four peaks, the multiplicative relative error is bounded by
\begin{equation}\label{eq:folded-error}
 \begin{array}{lll}
 r\in\rc:&6Q_*+24u_*/x,&1/2\le t\le2,\\
 r\in\ro:&5xQ_*+20u_*,&1/2\le t\le2,\\
 r\in\rc:&5Q_*+20u_*/x,&2\le t\le\alpha\sqrt n,\\
 r\in\ro:&5xQ_*+20u_*,&2\le t\le\alpha_0\sqrt n.
 \end{array}
\end{equation}
In the third line, the comparison amplitude is $-e^{-5z}\mathcal B_r(e^{-5z})$.
\end{lemma}
\begin{proof}
Apply shifted summation on $[1,\infty)$ to the tail quotient in
\eqref{eq:tail-layers}. The upper endpoint vanishes and the first two lower
endpoint terms give $\sum_sc_sf_s(1)$ and \eqref{eq:bh}.
The remaining error satisfies
\begin{align}
 |b_h|&\le\frac{2|z|x}{5(1-x)},\label{eq:tail-b}\\
 |R_{\mathrm{tail}}|&\le\frac{5x}{6n^2}
  \left[\frac{|z|^2}{(1-x)^2}+
        \frac{|z|^3(1+x)}{t(1-x)^3}\right]
       <\frac{4|z|^2x}{n^2}.\label{eq:tail-R}
\end{align}
The last comparison uses $|z|/t\le\sqrt5<9/4$ and $x\le1/12$.
In \eqref{eq:primitive-modular}, the transformed product ratio has
logarithmic error $\le2u_*$ if $u_*\le1/100$.
The tail integral changes $cn/z$ into $nA(z)$.
Thus the total logarithmic error is $\le|b_h|/n+4|z|^2x/n^2+2u_*$,
which is $<1/10$ on the stated domain.
Twice this bound is at most $xQ_*+4u_*$.
For the domain check, $t/n< (13/50)/\sqrt n<1/100$ and
$|z|/n<(9/4)(13/50)/41$. Also $u_*<e^{-40}$ on the largest ordinary arc.
These bounds suffice for the logarithmic error condition, even though
$|z|/n$ need not itself be $\le1/100$.

For $t\in[1/2,2]$, Lemma~\ref{lem:disk} bounds the folded error in a
cancelling class by
\[
 \frac{4(12/11)}{203/275}(Q_*+4u_*/x)
      =\frac{1200}{203}Q_*+\frac{4800}{203}u_*/x
      <6Q_*+24u_*/x.
\]
The ordinary classes have real and complex amplitude modulus $>1$, giving
$5xQ_*+20u_*$.
For $t\ge2$, $x<1/20000$, $|\lambda_h|<100/99$,
$|\mathcal B_r(X)|\ge1-4x>99/100$, and the ratio of complex to real
amplitudes is $<101/100$. These give the last two lines of
\eqref{eq:folded-error}.
Notice that division by $x$ is made only after the $x$ in the single-peak
Euler--Maclaurin error has been retained.
\end{proof}

On the boundary arcs, Lemma~\ref{lem:disk} also yields the useful bounds
\begin{equation}\label{eq:boundary-amplitude}
 \begin{split}
 \left|e^{\ii v/(6n)}\frac{\mathcal B_r(e^{-5(t-\ii v)})}{\mathcal B_r(x)}-1\right|
    &\le21x|v|+\frac{|v|}{6n}\quad(r\in\rc),\\
 \left|e^{\ii v/(6n)}\frac{\Gamma_r(t-\ii v)}{\Gamma_r(t)}-1\right|
    &\le20x|v|+\frac{|v|}{6n}\quad(r\in\ro).
 \end{split}
\end{equation}
These follow by first estimating the amplitude difference, and then writing
$e^{\ii v/(6n)}H-1=e^{\ii v/(6n)}(H-1)+(e^{\ii v/(6n)}-1)$.

\section{Uniform phase integrals}\label{sec:phase}
The exponent used after the appropriate saddle point has been chosen is
\begin{equation}\label{eq:F}
 F(v)=n[A(t-\ii v)-A(t)+\ii A'(t)v],\quad
 U=\sqrt{nV(t)}\,v,\quad k=cn/t\quad(t>0).
\end{equation}
All normalized integrals below use the factor $\sqrt{nV/(2\pi)}$.
The symmetry $F(-v)=\overline{F(v)}$ will be used before taking absolute values.

\begin{lemma}\label{lem:central-phase}
For $0\le t\le1/2$, on $|v|\le1/20$, write
$F=-U^2/2+R+\ii\Psi$, with real $R,\Psi$. Then
\begin{equation}\label{eq:central-phase-bounds}
 |R|\le\frac{175U^4}{4nV},\qquad
 |\Psi|\le\frac{35|U|^3}{6\sqrt{nV}},\qquad |R|\le\frac7{64}U^2.
\end{equation}
The inner phase integral differs from the full normalized Gaussian integral
by at most $1000/(nV)$ plus its Gaussian truncation tail.
The additional error from $H_r$ in \eqref{eq:H-center} is at most $1730/(nV)$.
\end{lemma}
\begin{proof}
Taylor's formula and Lemma~\ref{lem:complex-cumulants} prove
\eqref{eq:central-phase-bounds}. The exponential envelope is
$e^{-25U^2/64}$.
After pairing $v$ and $-v$,
\[
 |e^R\cos\Psi-1|\le(|R|+\Psi^2/2)e^{\max(R,0)}.
\]
The fourth and sixth Gaussian moments give the constant
\[
 \frac{175}{4}\,3\left(\frac{32}{25}\right)^{5/2}
 +\frac{1225}{72}\,15\left(\frac{32}{25}\right)^{7/2}<1000.
\]
For the amplitude, $H_r=1+\ii\ell v+R_H$ and its first odd term contributes
only $-\ell v\sin\Psi$ to the real paired integral.
The fourth and second moments give
\[
 \frac{40(35/6)\,3\cdot2+220(3/2)}{nV}=\frac{1730}{nV}.
\]
The bounds $(32/25)^{5/2}<2$, $(32/25)^{7/2}<5/2$ and
$(32/25)^{3/2}<3/2$ are rational squared comparisons.
\end{proof}

For later use, on the whole central arc,
\begin{equation}\label{eq:central-loss}
 \Rea F(v)\le-\frac16nVv^2\quad(|v|\le1/5),\qquad
 \Rea F(v)\le-\frac38nVv^2\quad(|v|\le1/10).
\end{equation}
To see this, set $w_m(X)=\sum_jp_jp_{j+m}$.
The elementary characteristic function bound is
\begin{equation}\label{eq:weight-loss}
 \log D(X)-\log|D(Xe^{\ii\phi})|
     \ge\sum_{m=1}^4w_m(X)(1-\cos m\phi).
\end{equation}
The $m=4$ portion of the variance is at most $16/50<1/3$, since
$P_2(X)\ge50X^3$ in \eqref{eq:K2}.
The function $(1-\cos s)/s^2$ decreases on $[0,2\pi]$.
For $|v|\le1/5$, use $1-\cos3>49/25$ and $1-\cos4>8/5$;
for $|v|\le1/10$, use $1-\cos(3/2)>37/40$ and $1-\cos2>7/5$.
The resulting coefficients are bounded below by
$(2/3)(49/225)+(1/3)(1/10)>1/6$ and
$(2/3)(37/90)+(1/3)(7/20)>3/8$.
This proves \eqref{eq:central-loss}. In particular, the whole unweighted
normalized absolute phase integral is $<\sqrt3<2$.

\begin{lemma}\label{lem:middle-phase}
For $1/2\le t\le2$ and $n\ge N$, the inner arc $|v|\le t/4$ has phase
error less than $85/k$ plus the Gaussian truncation tail.
Its normalized absolute integral is $<29/25$, and its normalized absolute
first moment is $<16/(15\sqrt{nV})$.
For either amplitude $H$ in Lemma~\ref{lem:disk}, the weighted outer arc
$t/4\le|v|\le2t$ contributes less than
\begin{equation}\label{eq:middle-outer}
 \frac{28e^{-2k/85}}{\sqrt k}+4\sqrt k\,e^{-k/20}.
\end{equation}
The unweighted normalized absolute integral on $|v|\le2t$ is $<6/5$.
\end{lemma}
\begin{proof}
Lemma~\ref{lem:derivatives} gives on the inner arc
\[
 |R|\le\frac{23U^4}{12k},\quad
 |\Psi|\le\frac{11|U|^3}{6\sqrt k},\quad
 |R|/U^2\le\frac{46}{384}<\frac18.
\]
The last bound is obtained directly from $|A''''|v^2/(24V)$, not by
combining the previous bounds with an upper variance estimate.
Thus the envelope is $e^{-3U^2/8}$ and the moment constant is
\[
 \frac{23}{12}\,3(4/3)^{5/2}
       +\frac{121}{72}\,15(4/3)^{7/2}<85.
\]
The absolute integral is at most $\sqrt{4/3}<29/25$, and the absolute
first moment is at most
\[
 \frac{8}{3\sqrt{2\pi}\sqrt{nV}}<\frac{16}{15\sqrt{nV}},
\]
using $\sqrt{2\pi}>5/2$.

For $|y|=|v|/t\le1/2$, the second-derivative bound
$t^3|R_A''|<(24/25)c$ gives
\[
 \Rea F\le-\frac25 k\frac{y^2}{1+y^2}.
\]
For $1/2\le|y|\le2$, the last assertion of Lemma~\ref{lem:derivatives}
gives $\Rea F\le-k/20$.
On $1/4\le|y|\le1/2$, $y^2/(1+y^2)\ge(4/17)|y|$.
Integrating the resulting positive exponential envelope, and using
$nVt^2\le3k$ and $|H|<7/4$, gives the first term of
\eqref{eq:middle-outer}; bounding the remaining length gives the second.
The envelope is extended to infinity only as an upper bound for a positive
integral. The phase inequality is used solely on its stated domain.
At $k>455/2$ the right side of \eqref{eq:middle-outer} is $<.012$.
Consequently the unweighted absolute integral is $<29/25+.012<6/5$.
\end{proof}

\begin{lemma}\label{lem:boundary-phase}
For $t\ge2$, let $k=cn/t>0$. On $|v|\le2t$,
\begin{equation}\label{eq:boundary-loss}
 \Rea F(v)\le-\frac{49}{50}k\frac{(v/t)^2}{1+(v/t)^2}.
\end{equation}
The normalized unweighted phase integral differs from $1$ by at most
\begin{equation}\label{eq:Gnew}
 G_*(k)=\frac9{4k}+\frac{5e^{-49k/850}+(23/10)e^{-199k/3200}}{\sqrt k}.
\end{equation}
Its normalized absolute integral is at most
$\sqrt{15/14}+5e^{-49k/850}/\sqrt k$.
In particular, for $k\ge42$ these absolute integrals are $<6/5$,
$G_*(42)<.16$, and $G_*(66)<.06$.
\end{lemma}
\begin{proof}
The remainder estimate
$t^3|R_A''|\le7t^2x/(1-x)$, $x=e^{-5t}$, and the rational inequality
\[
 \frac{7}{2}\frac{2^2(1/20000)\cdot5}{1-1/20000}<\frac c{50}
\]
prove \eqref{eq:boundary-loss}; $t^2x/(1-x)$ decreases for $t\ge2$.
This stronger comparison is used here rather than the coarser variance
interval alone.
On the inner arc $|v|\le t/4$, Lemma~\ref{lem:derivatives} gives
\[
 |R|\le\frac{41}{40(199/100)^2}\frac{U^4}{k},\qquad
 \frac{\Psi^2}{2}\le\frac{(61/60)^2}{2(199/100)^3}\frac{U^6}{k},\qquad
 |R|/U^2<\frac1{30}.
\]
The envelope is $e^{-7U^2/15}$.
The fourth and sixth moments give a total constant $<9/4$:
\[
 3\frac{41}{40(199/100)^2}(15/14)^{5/2}
 +15\frac{(61/60)^2}{2(199/100)^3}(15/14)^{7/2}<\frac94.
\]
On $1/4\le|y|\le2$, \eqref{eq:boundary-loss} and
$y^2/(1+y^2)\ge(4/17)|y|$ imply the positive envelope
$e^{-(98/425)k|y|}$.
Since $nVt^2\le201k/100$, its normalized integral is
$<5e^{-49k/850}/\sqrt k$.
The Gaussian truncation point satisfies $U_0^2\ge199k/1600$, giving the
double tail $<(23/10)e^{-199k/3200}/\sqrt k$.
These prove \eqref{eq:Gnew}.
All terms in $G_*$ decrease with $k$; the endpoint comparisons are rational
exponential certificates.
\end{proof}

\section{Bounds on the complementary arcs}\label{sec:minor}
We require estimates on the whole complementary contour, including points
near roots of unity whose order may depend on $n$.

\subsection{Moving rational cusps}
\begin{lemma}\label{lem:global-minor}
Let $0<\delta\le1/100$. If $\theta$ has distance at least $2\delta$ from
the four primitive fifth roots, then
\begin{equation}\label{eq:infinite-minor}
 |S_\infty(e^{-\delta+\ii\theta})|
       \le2e^{c/(4\delta)-\delta/6}.
\end{equation}
Consequently, for $n\ge N$, $\delta=t/n$, and the same distance condition,
\begin{align}
 |S_n(e^{-t/n+\ii\theta})|&\le256e^{nA(t)-(c/3)n/t}\quad(t\ge1/2),\label{eq:middle-minor}\\
 |S_n(e^{-t/n+\ii\theta})|&\le2e^{nA(t)-(37c/50)n/t}\quad(t\ge2).\label{eq:boundary-minor}
\end{align}
The displayed bounds are used only when $t/n\le1/100$.
\end{lemma}
\begin{proof}
For a reduced cusp $a/m$, put $g=\gcd(m,5)$,
$q=e^{2\pi\ii a/m-w}$ and $\Rea w=\delta$.
Applying \eqref{eq:eta-transform} to $\eta(\tau)$ and $\eta(5\tau)$ gives
\begin{equation}\label{eq:general-cusp}
 |S_\infty(q)|=\sqrt{\frac5g}
 \exp\left[\frac{\pi^2}{6m^2}\left(\frac{g^2}{5}-1\right)\Rea\frac1w-\frac\delta6\right]
 \frac{|P(u)|}{|P(v)|},
\end{equation}
where $|u|=e^{-\sigma}$, $|v|=e^{-(g^2/5)\sigma}$ and
$\sigma=4\pi^2\Rea(1/w)/m^2$.
Indeed, the transformed arguments are
$a_*/m+2\pi\ii/(m^2w)$ and
$b_*/(m/g)+2\pi\ii/[5(m/g)^2w]$.
The square root ratio has modulus $\sqrt{5/g}$; the eta multipliers have
modulus one. Their leading exponentials and $|q^{1/6}|=e^{-\delta/6}$
give \eqref{eq:general-cusp}.

Set $Q_0=\lceil\sqrt{2\pi/\delta}\rceil$.
Dirichlet approximation supplies $m\le Q_0$ and
$|y|=|\theta-2\pi a/m|\le2\pi/(mQ_0)$.
With $w=\delta-\ii y$,
\[
 \frac{\Rea(1/w)}{m^2}\ge
 \frac\delta{Q_0^2\delta^2+4\pi^2/Q_0^2}\ge\frac1{5\pi}.
\]
Here $\sqrt{2\pi/\delta}>20$ and
$Q_0\le(21/20)\sqrt{2\pi/\delta}$.
Thus $|u|<1/12$.
For any radii $p,s<1$,
\[
 |P(u)|\le e^{p/(1-p)},\qquad |P(v)|^{-1}\le e^{s/(1-s)^2}.
\]
If $g=1$, then $|v|<8/13$ and the square root and product ratio in
\eqref{eq:general-cusp} are $<256$.
Its exponential coefficient is negative. Since $c/(4\delta)>13/2>5$
and $2e^5>256$, this case is bounded by the right side of
\eqref{eq:infinite-minor}.
If $g=5$, then $|v|<1/250000$ and the product ratio is $<6/5$.
For $m=5\ell$, $\ell\ge2$, the positive exponent is at most
$c/(4\delta)$; for $m=5$ the distance hypothesis gives $|y|\ge2\delta$
and hence an exponent at most $c/(5\delta)$.
This proves \eqref{eq:infinite-minor} uniformly in the moving denominator.

The finite tail quotient is bounded by its positive real value, and
\[
 \log T_n(e^{-\delta})\le\frac{21x}{25\delta(1-x)},\qquad
 c/\delta-nA(t)\le\frac{x}{5\delta(1-x)},\quad x=e^{-5t}.
\]
The first bound uses
$1/(1-e^{-5\delta})\le1+1/(5\delta)$ and $\delta\le1/100$.
Their combined cost is $26x/[25\delta(1-x)]$.
For $t\ge2$, it is $<c/(100\delta)$; for $t\ge1/2$, it is
$\le26/(275\delta)<5c/(12\delta)$.
Subtract these costs from the gap $3c/(4\delta)$.
For the middle estimate we retain the larger harmless prefactor $256$.
\end{proof}

\subsection{A root-one integral}
For $0<\rho<1$, define
\begin{equation}\label{eq:I}
 I_\rho(a,b)=\int_0^1\log|1-\rho e^{-au+\ii bu}|\dd u,
 \qquad C_\rho=\Li(\rho).
\end{equation}
\begin{lemma}\label{lem:root-one}
For $\rho\ge99/100$, $a\ge0$ and $b\in\RR$, $I_\rho(a,b)\le9/40$.
For every $0<\rho<1$ and $a>0$,
\begin{equation}\label{eq:I-square}
 I_\rho(a,b)\le\frac{\Li(\rho e^{-a})^2}{4C_\rho a}.
\end{equation}
For $\rho=999/1000$, uniformly in real $y$,
\begin{equation}\label{eq:root-one-gap}
 A(t)-4I_\rho(5t,5y)>
 \begin{cases}4/15,&0\le t\le1/10,\\21/50,&1/10\le t\le1/2.\end{cases}
\end{equation}
\end{lemma}
\begin{proof}
Let $g_\rho(v)=\log|1-\rho e^{\ii v}|$.
It is even and periodic, has zero half-period integral, and increases on
$[0,\pi]$. Thus its prefix integral on this interval is nonpositive.
For $T\in[\pi,2\pi]$, put $s=2\pi-T$. The integral up to $T$ equals the
negative integral up to $s$.
The inequality
$|1-\rho e^{\ii v}|^2\ge\rho|1-e^{\ii v}|^2$ and $\log\rho\ge-1/99$ give
\[
 -\int_0^s g_\rho(v)\dd v\le s(1-\log s)+s^3/42+s/198.
\]
Here $\log|1-e^{\ii v}|\ge\log v-v^2/14$ on $(0,\pi]$,
by the cubic sine lower bound and $\pi^2<10$.
To bound the last expression by $(9/40)(2\pi-s)$, let
$f(s)=s(49/40-\log s)+s^3/42+s/198$.
For $s\le9/5$,
\[
 f(s)\le e^{9/40}+243/1750+1/110<45/32<9\pi/20.
\]
On $[9/5,\pi]$, $f'(s)=9/40-\log s+s^2/14+1/198$ is convex and is
negative at both endpoints, using $\log(9/5)>1/2$, $\log\pi>1$ and
$\pi^2<10$. Hence $f$ decreases there.
Removing complete periods extends the average bound to arbitrary $T$.
It follows that $I_\rho(0,b)\le9/40$.

The function
$F_\rho(z)=[\Li(\rho e^{-z})-C_\rho]/z$ is analytic in the right
half-plane, continuous at $z=0$, and tends uniformly to zero at infinity.
Its real part at $a-\ii b$ is $I_\rho(a,b)$.
The maximum principle on expanding half-disks proves the first assertion.
For the second, with $L=\Li(\rho e^{-a})$ and
$r=(a^2+b^2)^{-1/2}$, bound the real part by
$-C_\rho ar^2+Lr\le L^2/(4C_\rho a)$.

For $t\le1/10$, Jensen's inequality gives
$A(t)\ge\log D(e^{-5t/2})>\log D(7/9)>7/6$,
which gives the first line of \eqref{eq:root-one-gap}.
For the second, $C_{999/1000}>8/5$, since
\[
 \pi^2/6-C_\rho\le(1-\rho)(1-\log(1-\rho))/\rho<8/999.
\]
Equation~\eqref{eq:I-square} gives
$4I_\rho(5t,5y)<\Li(e^{-5t})^2/(8t)$.
On each of the forty intervals $[l,r]=[1/10+j/100,1/10+(j+1)/100]$,
the rational comparison
\begin{equation}\label{eq:forty}
 \frac{13/50-\Li(e^{-5r})/5}{r}
           -\frac{\Li(e^{-5l})^2}{8l}>\frac{21}{50}
\end{equation}
holds with the specified series bounds in the appendix.
Use monotonicity of $A$ and of $\Li(e^{-5t})^2/t$ to cover each full interval.
\end{proof}

\subsection{The central complementary contour}
\begin{lemma}\label{lem:central-minor}
For $n\ge N$ and $0\le t\le1/2$, if
$\dist(\theta,\{2\pi h/5:1\le h\le4\})\ge1/(5n)$, then
\begin{equation}\label{eq:central-minor}
 |S_n(e^{-t/n+\ii\theta})|\le e^{nA(t)-\gamma n}.
\end{equation}
\end{lemma}
\begin{proof}
We divide the circle into a far arc, primitive-root neighborhoods of half-width
$5/n$, and a root-one neighborhood of half-width $20/n$.
The neighborhoods are disjoint for $n\ge N$.

\emph{Far arc.}
Let $q=e^{-t/n+\ii\theta}$ and
\[
 m_n(q)=\sum_{j,s}q^{5j-s}=(1-q^{5n})\left[\frac1{1-q}-\frac1{1-q^5}\right].
\]
Outside the neighborhoods just specified,
\begin{equation}\label{eq:far-m}
 \frac{|m_n(q)|}{4n}<\frac1n+\frac{1001\pi}{100000}.
\end{equation}
If the distance to $1$ is at most $\pi/5$, the nearest fifth root is $1$,
and the geometric sum bound gives the second term alone, using distance
$20/n$ and $e^{3/(4n)}<1001/1000$.
Otherwise $|\sin(\theta/2)|\ge\sin(\pi/10)>3/10$, and the displayed
decomposition bounds its first term by $1/n$.
The second term uses $|\sin(5\theta/2)|\ge25/(\pi n)$ and
$e^{5/(4n)}<1001/1000$.
For $H(t)=(1-e^{-10t})/(10t)$, with $H(0)=1$, shifted rectangular summation
gives $(4n)^{-1}\sum_{j,s}|q|^{2(5j-s)}\le H(t)+1/n$.
Arithmetic--geometric mean applied to the squared factor moduli yields
\begin{equation}\label{eq:far-log}
 \log|S_n(q)|\le2n\log[1+H(t)+3/n+2(1001/1000)\pi/100].
\end{equation}
On the fifty intervals $[j/100,(j+1)/100]$, use $A$ at the right endpoint
and $H$ at the left endpoint. Lower bounds for $A(r)$ are
$\log D(e^{-5r/2})$ if $r\le1/10$, and
$[13/50-\Li(e^{-5r})/5]/r$ otherwise.
The resulting rational gaps in \eqref{eq:far-log} all exceed $\gamma$.
This proves \eqref{eq:central-minor} on the entire far arc.

\emph{Primitive-root integral losses.}
Write $\theta=2\pi h/5+y/n$, $1\le h\le4$.
The four logarithmic factor integrals add to
$\int_0^1\log|D(e^{-5tu+5\ii yu})|\dd u$.
Use \eqref{eq:weight-loss} with
\[
 w_m(X)=\frac{X^m\sum_{j=0}^{4-m}X^{2j}}{D(X)^2}.
\]
The functions $w_2,w_3,w_4$ increase on $[0,1]$;
$w_1\le1/5$ has a single internal maximum.
For $1/13\le X\le1$, $w_1\ge13/200$ and $w_2\ge1/200$.
For $u\le1/2$, $X=e^{-5tu}>1/4$, and stronger bounds are
$w_1\ge3/20$, $w_2\ge3/80$, $w_3\ge9/1000$, $w_4\ge1/500$.
These facts follow from the polynomial certificates in
Appendix~\ref{app:certificates}.

For $1/5\le|y|\le3/10$, partition $u$ into twenty intervals $[a,b]$.
Use $\min(w_1(e^{-5b/2}),4/25)$ for $m=1$, and
$w_m(e^{-5b/2})$ for $m\ge2$, replacing exponentials by rational lower bounds.
The angle belongs to $[ma,(3/2)mb]\subset[0,6]$.
Its cosine maximum is at an endpoint. A twenty-cell lower integral sum,
with weights rounded downward to $10^{-7}$ and cosine upper bounds of degree
$28$, gives an angular loss $>69/2000$.
For $3/10\le|y|\le1/2$, the full-interval loss is at least $1/40$.
The stronger weights on $u\le1/2$ add the three terms
\[
 (3/20-13/200)/25+(3/80-1/200)/6+(9/1000)(5/18),
\]
and the fourth angle adds $1/1500$.
For $1/2\le|y|\le1$, the full-interval loss is $43/1000$;
the four half-interval integrals are bounded below by
$1/10,3/10,11/30,2/5$.
Combining them gives the following uniform raw losses:
\begin{equation}\label{eq:primitive-losses}
 \begin{array}{c|ccc}
 |y| &[1/5,3/10]&[3/10,1/2]&[1/2,1]\\ \hline
 \text{loss}&69/2000&2219/60000&1307/20000.
 \end{array}
\end{equation}
For example, the middle loss is
$1/40+(3/20-13/200)/25+(3/80-1/200)/6+(9/1000)(5/18)+1/1500$.
The half-interval integral is
$1/2-\sin(5my/2)/(5my)$; monotonicity of sinc on its first positive interval
and $\sin s\le1$ prove the bounds just used.

For $1\le|y|\le5$, put $W=\sum_mw_m=[1-D(X^2)/D(X)^2]/2$.
This increases in $X$, since
\[
 \frac{\mathrm d}{\mathrm dX}\log\frac{D(X^2)}{D(X)^2}
  =\frac2X(\EE_{X^2}R-\EE_XR)\le0.
\]
A hundred-cell lower sum yields
\[
 \int_0^1W(e^{-5u/2})\dd u>.23108294>231/1000.
\]
Integration by parts gives, for $m\ge2$, the cosine integral bound
$w_m(1)/(5m|y|)$.
For $m=1$, its total variation plus terminal value is at most
$2/5-w_1(1)=6/25$.
Since $w_m(1)=(5-m)/25$, their sum is at most $101/(1500|y|)$.
The raw angular loss is consequently at least
\begin{equation}\label{eq:primitive-long}
 231/1000-101/(1500|y|).
\end{equation}

\emph{Variation for $t\ge1/10$ and $1/5\le|y|\le1$.}
The unsmoothed functions
$g_s(u)=\log|1-\zeta^{-hs}e^{-5(t-\ii y)u}|$ have no singularity.
Shifted rectangular summation has error at most $\sum_s\TV(g_s)$.
Set $u_0=\pi/(25|y|)$.
On the initial interval, the four phases are a permutation of $2\pi j/5$,
and drift by at most $\pi/5$.
The inequality
\[
 \frac r{|1-re^{\ii\phi}|}\le\frac1{2\sin(\phi/2)}\quad(0\le r\le1)
\]
follows by squaring and reducing to
$(1-r)(1+r-2r\cos\phi)\ge0$.
Hence the initial total variation is at most
\[
 C_0\sqrt{1+(t/y)^2},\qquad
 C_0=\sum_{j=1}^4\int_0^{\pi/5}
        \frac{\dd s}{2\sin(\pi j/5+s/2)}<19/10.
\]
The last strict bound is a forty-cell upper sum with rational sine lower
bounds. Negative $y$ gives the same sum by reflection of the phases.
On $u\ge u_0$, the radius is at most $R=e^{-\pi t/(5|y|)}$.
Radial and angular variation then give the bound
$(12+20|y|/\pi)\log[(1+R)/(1-R)]$.
The upper bound function
\[
 F_y(s)=\frac{19}{10}\sqrt{1+s^2}
 +(12+20|y|/\pi)\log\coth(\pi s/10),\quad s=t/|y|,
\]
decreases for $0<s\le5/2$:
its derivative is less than
$19/10-(12+20|y|/\pi)(\pi/5)/\sinh(\pi s/5)<0$.
Use $e^{11/7}<5$ and $\pi s/5\le\pi/2<11/7$ for the last comparison.
Replacing $t$ by $1/10$ in this upper bound, and bounding the two terms on
each $y$ interval, gives total variation bounds $34,45,66$.
Subtracting these divided by $N$ from \eqref{eq:primitive-losses} leaves
more than $\gamma$; the tightest difference above $\gamma$ is
\begin{equation}\label{eq:variation-margin}
 2219/60000-45/1750-9/800=1/52500>0.
\end{equation}

\emph{Smoothing for the remaining neighborhoods.}
For $|w|\le1$ and $0<\rho<1$,
\[
 \log|1-w|\le\log|1-\rho w|+\log\frac2{1+\rho}.
\]
For the four smoothed factors, put $E=\log[(1+\rho)/(1-\rho)]$.
The radial variation of each factor is at most $E$, and its angular
variation over an angle of length $5|y|$ is at most $(5|y|/\pi+2)E$.
Thus the total variation is at most $(12+20|y|/\pi)E$.
For $\rho=999/1000$, the smoothing cost is $<.00201n$ and total variation
is $<92+49|y|$. Each probability pair in \eqref{eq:weight-loss} is reduced
by a factor no smaller than $\rho^8>99/100$.

If $t\le1/10$, then $X>3/5$, so $w_1\ge4/25$, $w_2\ge1/10$.
The raw losses on the three short intervals in \eqref{eq:primitive-losses}
are now at least $39/500$, $3/25$, $22/125$.
Their variation bounds are $107,117,141$.
Multiplying each loss by $.99$ and subtracting the variation divided by
$N$ and $.00201$ still leaves more than $\gamma$.
For $1\le|y|\le5$, the same smoothing applied to
\eqref{eq:primitive-long} leaves the concave function
\[
 .99[231/1000-101/(1500|y|)]-(92+49|y|)/N-.00201.
\]
Both endpoints exceed $\gamma$.

At root one, $h=0$, the integral is $4I_\rho(5t,5y)$.
For $t\le1/10$, $|y|\le3$, use the loss $4/15$ of
Lemma~\ref{lem:root-one} and variation $239$.
For $t\ge1/10$, $|y|\le5$, use the loss $21/50$ and variation $337$.
Again subtract $.00201$ and the variation divided by $N$.
For larger $|y|$, choose $\rho=49/50$ instead.
Now the smoothing cost is $<.0403n$ and total variation is
$<276/5+(147/5)|y|$, since $E=\log99<23/5$.
No use is made here of the first assertion of Lemma~\ref{lem:root-one}.
The dilogarithm formula implies
$|4I_\rho(5t,5y)|<8/(3|y|)$ for every $\rho$, and, when $t\ge1/10$,
$4I_\rho(5t,5y)<3/(5|y|)$ by the nonpositive constant part of its numerator.
Using $A(t)>7/6$ for $t\le1/10$ and $A(t)>12/25$ otherwise gives
the raw losses $7/6-8/(3|y|)$ on $[3,20]$ and
$12/25-3/(5|y|)$ on $[5,20]$.
Subtract $.0403+[276/5+(147/5)|y|]/N$.
Both resulting functions are concave and their endpoint values exceed
$\gamma$.

Every comparison in this proof either holds pointwise on an interval or
is an endpoint comparison justified by monotonicity or concavity.
The listed rational interval certificates are reproduced by the accompanying
program. These regions cover the full complementary contour.
At $t=0$, zero factors are handled by modulus zero and the smoothed bounds
remain valid. This completes the proof.
\end{proof}

\section{Coefficient extraction and completion of the analytic argument}
\label{sec:assembly}
Let $r=d\bmod5$. Cauchy's formula on the circle of radius $e^{-t/n}$ is
\begin{equation}\label{eq:cauchy}
 a_d(n)=\frac{e^{td/n}}{2\pi}
 \int_{-\pi}^{\pi}S_n(e^{-t/n+\ii\theta})e^{-\ii d\theta}\dd\theta.
\end{equation}
It remains valid at $t=0$, since $S_n$ is a polynomial.
Around $\theta_h=2\pi h/5$, write $\theta=\theta_h+v/n$ and $z=t-\ii v$.
The central arcs have $|v|\le1/5$; the other arcs have $|v|\le2t$.
After summing the four arcs, the reference quantities are
\begin{align}
 M_d&=\frac{e^{nA(t)+td/n-t/(6n)}\Gamma_r(t)}
              {\sqrt{2\pi n^3V(t)}},
 &&d=n^2\beta(t),\label{eq:mainterm-o}\\
 M_d&=-\frac{e^{nA(t)+t(d/n-5)-t/(6n)}\mathcal B_r(e^{-5t})}
              {\sqrt{2\pi n^3V(t)}},
 &&d-5n=n^2\beta(t),\quad r\in\rc.\label{eq:mainterm-c}
\end{align}
The first expression is used for every residue in the center, and for
$r\in\ro$ elsewhere. The second is used for $r\in\rc$ when $t\ge1/2$.
In both cases $\epsd M_d>0$ by Lemma~\ref{lem:amplitude-sign}.
The exact factorization $\Gamma_r(z)=-e^{-5z}\mathcal B_r(e^{-5z})$
moves the factor $e^{5\ii v}$ into the exponential in the second case.
Thus the linear term cancels precisely for the saddle equation in
\eqref{eq:mainterm-c}; using the unshifted equation there would leave a
nonzero oscillatory linear term.

We normalize the local integral by $\sqrt{nV/(2\pi)}$.
The reference Gaussian integral is then one. Errors below are relative to
$|M_d|$. A bound strictly less than one implies $\epsd a_d(n)>0$.
We retain separate bounds for the inner phase, amplitude, local expansion,
outer part of the four arcs, and complementary contour.

\subsection{The central range}
Suppose $0\le t\le1/2$ and $n\ge N$.
Then $nV\ge3500$ and the inner interval $|v|\le1/20$ has Gaussian
endpoint $U_0^2\ge35/4$. The phase bound of
Lemma~\ref{lem:central-phase} and the Gaussian truncation give
\[
 \frac{1000}{nV}+\frac{2e^{-U_0^2/2}}{\sqrt{2\pi}U_0}<0.29.
\]
Here $U_0>2.9$, $e^{35/8}>79$, and $\sqrt{2\pi}>5/2$ suffice to
bound the second term by $0.004$.
The paired amplitude error is at most $1730/(nV)<0.495$.
The absolute phase integral over the whole central arc is less than two,
so the folded local error in \eqref{eq:central-EM-fold} contributes less than
\[
 2\frac{100}{7}\left(\frac2N+\frac{3250}{N^2}\right)<0.064.
\]

For clarity, the elementary envelope used on the remaining parts of the
central arcs is
\begin{equation}\label{eq:outer-integral}
 \sqrt{\frac{nV}{2\pi}}
 \int_{|v|\ge a}(1+C|v|)e^{-g nVv^2}\dd v
 \le\frac{e^{-g nVa^2}(1/a+C)}{g\sqrt{2\pi nV}}.
\end{equation}
Apply it with $(a,C,g)=(1/20,55,3/8)$ and $(1/10,229,1/6)$.
The bounds \eqref{eq:H-outer} and \eqref{eq:central-loss}, together with
$\sqrt{2\pi nV}>148$, $e^{105/32}>26$, and $e^{35/6}>340$,
give contributions less than $0.052$ and $0.029$, respectively.
Using infinite positive envelopes in \eqref{eq:outer-integral} only enlarges
the original finite integrals. Their total contribution is less than $0.081$.

By Lemma~\ref{lem:central-minor}, the remaining contour contributes at most
\begin{equation}\label{eq:center-minor-relative}
 200n^{3/2}e^{-\gamma n}<0.042.
\end{equation}
To verify the prefactor, $|\Gamma_r(t)|\ge7e^{-5t}/10>7/130$,
$\sqrt{2\pi V}<21/2$, and $e^{t/(6n)}<101/100$ imply
$(21/2)(130/7)(101/100)<200$.
The function in \eqref{eq:center-minor-relative} decreases for $n\ge N$.
At $N$, use $\sqrt N<42$ and $e^{\gamma N}>350000000$.
Consequently,
\begin{equation}\label{eq:central-total}
 \left|\frac{a_d(n)}{M_d}-1\right|
 <0.29+0.495+0.064+0.081+0.042=\frac{243}{250}<1.
\end{equation}

\subsection{The intermediate range}
Let $1/2\le t\le2$. Here $k=cn/t>455/2$, $nV>105$, and
$\sqrt{nV}>10$. On $|v|\le t/4$, Lemma~\ref{lem:middle-phase}
gives phase error $85/k$ and a Gaussian truncation less than $0.0002$.
Indeed $U_0^2\ge k/16>225/16$ and $e^{k/32}>1200$.
Together these are less than $0.375$.

The disk bounds in Lemma~\ref{lem:disk} give normalized amplitude
derivative less than $9/4$. For a cancelling residue, the bound follows from
\[
 \frac{5x(32/9)}{203/275}+\frac7{24n}
 \le\frac{44000}{1827}x+\frac7{24n}<\frac94,
 \qquad x=e^{-5t}<\frac1{12}.
\]
For an ordinary residue one may use $(1600/81)x+7/(24n)<9/4$.
The normalized first absolute moment is bounded by
$16/(15\sqrt{nV})$; hence the amplitude error is less than $0.25$.
This estimate uses the factor $\sqrt{2\pi}$ in the normalization, as
specified in Lemma~\ref{lem:middle-phase}.

On the whole four arcs the absolute phase integral is less than $6/5$,
and the complex-to-real amplitude ratio is less than $7/4$.
The folded local error in \eqref{eq:folded-error} therefore contributes at most
\[
 \frac65\frac74
 \left[6\left(\frac9{2N}+\frac{162}{N^2}\right)+24\cdot10^{-8}\right]
 <0.034.
\]
We used $|z|\le\sqrt5t<(9/4)t$ and $u_*/x<e^{-100}<10^{-8}$.
The weighted outer part is at most the expression in
\eqref{eq:middle-outer}, which is less than $0.012$ for $k>455/2$.
The complementary contour, by \eqref{eq:middle-minor}, contributes less than
\begin{equation}\label{eq:middle-minor-relative}
 150000000 n^{3/2}e^{-13n/300}<0.0001.
\end{equation}
For example, $|\Gamma_r(t)|\ge(2/3)x$ in the cancelling classes,
$x>1/25000$, $V\le3c/t^3$, $t\ge1/2$, and $e^{t/(6n)}<2$ give
prefactor at most
$256(9/4)3(3/2)2(25000)=129600000<150000000$.
The ordinary classes satisfy the same bound. The right side of
\eqref{eq:middle-minor-relative} is decreasing for $n\ge N$ and the
endpoint comparison follows from the rational exponential certificates.
We obtain
\begin{equation}\label{eq:middle-total}
 \left|\frac{a_d(n)}{M_d}-1\right|
 <0.375+0.25+0.034+0.012+0.0001=0.6711<1.
\end{equation}

\subsection{The boundary range}
For $r\in\rc$, take $2\le t\le\alpha\sqrt n$, where
$\alpha=\sqrt{5c/49}\in(4/25,1/6)$.
For $r\in\ro$, take $2\le t\le\alpha_0\sqrt n$, where
$\alpha_0=\sqrt c/2\in(1/4,13/50)$.
In the former case $k>66$; in the latter $k>42$.
Lemma~\ref{lem:boundary-phase} gives phase errors less than $0.06$ and
$0.16$, respectively, and absolute phase integrals less than $6/5$.
By \eqref{eq:boundary-amplitude}, the amplitude errors are less than
\[
 \frac65\left(84e^{-10}+\frac{\alpha}{3\sqrt n}\right)<0.007,
 \qquad
 \frac65\left(80e^{-10}+\frac{\alpha_0}{3\sqrt n}\right)<0.008.
\]
The disk amplitude ratio is less than $101/100$ in this range.
For a cancelling residue the local error is less than
\[
 \frac65\frac{101}{100}
 \left[\frac{15}{8\sqrt n}+\frac{50}{9n}+20e^{-40}\right]<0.06.
\]
Indeed $Q_*\le3/(8\sqrt n)+10/(9n)$ and $u_*/x<e^{-40}$.
For ordinary residues use $x<1/20000$,
$Q_*\le117/(200\sqrt n)+8/(3n)$, and $u_*<e^{-40}$ in
\eqref{eq:folded-error}. This contributes less than $0.001$.

The normalized version of \eqref{eq:boundary-minor} is bounded by
\begin{align}
 &\frac{10000}{128}n^{3/4}
       e^{-(563/250)\alpha\sqrt n} &&(r\in\rc),\label{eq:boundary-minor-c}\\
 &\frac{10000}{128}n^{3/4}
       e^{-(37/25)\sqrt c\sqrt n} &&(r\in\ro).\label{eq:boundary-minor-o}
\end{align}
The prefactor and monotonicity calculation is as follows.
We use $\sqrt{2\pi}<13/5$, which follows by squaring and using $\pi<22/7$.
For $t\ge2$, Lemma~\ref{lem:disk} and $x<1/20000$ give
$\Br(x)>999/1000$ in the cancelling classes and
$|\Gr(x)|>13/10$ in the ordinary classes.
Since $n\ge1750$, $t\le\alpha_0\sqrt n$, and $\alpha_0<13/50$,
$e^{t/(6n)}<101/100$. Thus the real radius/amplitude ratios are bounded
by $51/50$ and $4/5$, respectively.
At $t=\alpha\sqrt n$, use $c/\alpha^2=49/5$,
$\alpha^{-1/2}<5/2$ and
$\sqrt{(201/100)(49/5)}<9/2$. The cancelling prefactor is less than
\[
 2\frac{51}{50}\frac{13}{5}\frac92\frac52
 =\frac{5967}{100}<60.
\]
At $t=\alpha_0\sqrt n$, use $\alpha_0=\sqrt c/2$,
$c^{-1/4}<10/7$, $2^{3/2}<3$, and $\sqrt{201/100}<3/2$.
The ordinary prefactor is less than
\[
 2\frac45\frac{13}{5}\frac32\,3\frac{10}{7}
 =\frac{936}{35}<27.
\]
Both are less than the unchanged common envelope $10000/128$.
The accompanying boundary certificate checks these comparisons using rational arithmetic.
In the cancelling case the dependence on $t$ is
$t^{-3/2}e^{5t-(37c/50)n/t}$, which increases for $t\ge2$.
At $t=\alpha\sqrt n$ its exponential factor is
$e^{-(563/250)\alpha\sqrt n}$, since $c/\alpha^2=49/5$.
In the ordinary case the derivative of the logarithm is
$[-3/2+(37/50)k]/t>0$; use $t=\alpha_0\sqrt n$.
Both displayed bounds then decrease as functions of $n\ge N$.
At $N$, $N^{3/4}<275$, $\sqrt N>41$,
$e^{(563/250)(4/25)41}>2500000$, and
$e^{(37/25)\sqrt c\sqrt N}>e^{30}>10^{12}$.
Thus the two complementary contributions are less than $0.01$ and $0.0001$.
The totals are
\begin{equation}\label{eq:boundary-total}
 \left|\frac{a_d(n)}{M_d}-1\right|
 <\begin{cases}0.137,&r\in\rc,\\0.1691,&r\in\ro.\end{cases}
\end{equation}
All hypotheses of the local estimates hold: $t/n<(13/50)/41<0.01$,
the logarithmic correction is less than $0.1$, and the transformed-product
errors have the bounds just used. The condition $|z|/n<0.01$ is unnecessary;
the relevant small radial parameter is $t/n$.

\subsection{Joining the degree intervals}
\begin{proof}[Proof of Theorem~\ref{thm:main}]
By \eqref{eq:reciprocity} it suffices to consider $d\le5n^2$.
Write $b=\beta(2)$ and $B=\beta(1/2)$.
For the ordinary residues, the exact boundary in Section~\ref{sec:edge}
covers $d\le5n$. The analytic boundary starts at
\[
 n^2\beta(\alpha_0\sqrt n)<\frac{cn^2}{\alpha_0^2 n}=4n
\]
and reaches $bn^2$. The intermediate range reaches $Bn^2$, and the central
range reaches $5n^2$. Strict monotonicity of $\beta$ supplies a unique
saddle for every degree in these intervals.

For residue $3$ the exact boundary reaches $15n-2$, and for residue $4$
it reaches $15n-1$. The shifted analytic boundary starts below
\[
 5n+n^2\beta(\alpha\sqrt n)<5n+\frac{c}{\alpha^2}n=14.8n
 <15n-2.
\]
It reaches $5n+bn^2$, and the intermediate range then reaches $5n+Bn^2$.
The center starts at $Bn^2$, so the two ranges overlap by $5n$.
The inequalities $b>1/20$ and $B<16/15$ ensure that the ranges are nonempty
and their joins occur below $5n^2$.
There are no gaps from integer rounding, since the real intervals overlap.
The four root arcs are disjoint: their half widths are $1/(5n)$ in the
center and at most $2(13/50)/41$ otherwise.
The error bounds \eqref{eq:central-total}, \eqref{eq:middle-total}, and
\eqref{eq:boundary-total} are all strictly less than one. They give the
desired sign on each analytic interval; the exact boundary gives the
remaining degrees. Reflection completes the proof.
\end{proof}

\begin{table}[ht]
\centering
\caption{Relative error budgets for $n\ge1750$. The phase column includes
Gaussian truncation; the outer column refers to the remainder of the four
principal arcs.}
\label{tab:budgets}
\small
\begin{tabular}{@{}lrrrrrr@{}}
\toprule
Range & Phase & Amplitude & Local & Outer & Complement & Total\\
\midrule
Center & .290 & .495 & .064 & .081 & .042 & .972\\
Intermediate & .375 & .250 & .034 & .012 & .0001 & .6711\\
Boundary, $r\in\rc$ & .060 & .007 & .060 & included & .010 & .137\\
Boundary, $r\in\ro$ & .160 & .008 & .001 & included & .0001 & .1691\\
\bottomrule
\end{tabular}
\end{table}

\section{Exact finite verification}\label{sec:exact}
The finite calculation uses integer arithmetic throughout. It can either
retain the whole polynomial with an unconditional digit width, or exploit
reciprocity with an inductively justified smaller width. We describe the
algorithm, the author-confirmed finite result, and the available computation records.

\subsection{Capacity and signed digits}
Let $A_r(n)=\sum_{d\equiv r\ (5)}a_d(n)$.
Since $S_n(1)=0$ and $S_n(\zeta^h)=5^n$ for $1\le h\le4$, roots-of-unity
filtering gives
\begin{equation}\label{eq:residue-sums}
 A_0(n)=4\cdot5^{n-1},\qquad A_r(n)=-5^{n-1}\quad(1\le r\le4).
\end{equation}
If the signs have been verified at stage $n$, it follows that
$\norm{S_n}_1=8\cdot5^{n-1}$.
The next block has coefficient $\ell^1$ norm at most $16$. Thus, before
checking stage $n\ge2$, the already verified preceding stage implies
\begin{equation}\label{eq:capacity}
 |a_d(n)|\le\norm{S_n}_1\le128\cdot5^{n-2}.
\end{equation}
For $n=1$, use the unconditional bound $16$.
This induction is not an assumption about the signs currently being tested:
the program stops on a violation and only advances from a verified stage.
For target $T$, choose a power of two $B$ with
$B\ge4\max(16,128\cdot5^{T-2})$ when $T\ge2$.
The choice $B=2^{4096}$ suffices for $T=1749$ by exact integer comparison.

Write the ordinary unsigned base-$B$ digits of the packed evaluation as
$u_0,u_1,\ldots$. Under the bound $|a_d|\le B/4$, the outgoing borrow is
$b_d=\mathbf1_{\{u_d\ge B/2\}}$, and
\begin{equation}\label{eq:decode}
 b_{-1}=0,\qquad a_d=u_d+b_{d-1}-Bb_d.
\end{equation}
Indeed $u_d$ belongs to the lower digit range for a nonnegative coefficient
and to the upper range for a negative coefficient, with the sole zero
exception caused by an incoming borrow. More precisely, $a_d=0$ exactly
when either $u_d=0$ and $b_{d-1}=b_d=0$, or $u_d=B-1$ and
$b_{d-1}=b_d=1$. All bytes of a candidate zero digit are checked.
Consequently the test is exact; no modular sign inference or floating point
coefficient approximation is involved.

\subsection{Reciprocal prefixes and updates}
At stage $n$ it suffices to decode the prefix $0\le d\le H_n=5n^2$.
Before multiplying the next block, first extend the \emph{old} polynomial
to the new cutoff using \eqref{eq:reciprocity}; coefficients beyond the old
degree are zero. Reversing unsigned digits alone would be incorrect.
If the source digit is $u$, its incoming and outgoing borrows are $p$ and
$b$, and the new incoming borrow is $c$, its coefficient is $u+p-Bb$.
The destination digit is therefore
\begin{equation}\label{eq:reflect-digits}
 u_{\rm new}=u+p-c\pmod B.
\end{equation}
On overflow the new outgoing borrow is $b-1$; on underflow it is $b+1$;
otherwise it is $b$. The coefficient bound ensures that it remains in
$\{0,1\}$. All source positions lie strictly below the old central degree,
so this extension can be performed in place.

Multiplication by a polynomial with nonnegative exponents cannot bring a
high coefficient into a lower position. Hence truncation after every factor
is legitimate. Evaluation at $B$ implements it as reduction modulo
$B^{H_n+1}$. The supplied kernel uses four factors $B^{5n-s}-1$ in a
block; their four minus signs cancel. Digits are decoded after the complete
block, when \eqref{eq:capacity} applies. No bound on the intermediate
signed polynomial is required for the modular integer operations.

The native implementation stores a single buffer of 64-bit limbs. For a
factor with shift $k$ limbs, multiplication is the exact modular operation
$X\mapsto(2^{64k}X-X)\bmod2^{64L}$. It processes descending bands of
length at most $k$, ensuring that each shifted source is read before it is
overwritten. Subtraction borrow propagation is determined from the original
band before the band is written; the parallel version places a barrier
between these steps. A portable implementation and a GMP implementation
are supplied, together with comparison tests against ordinary Python
integers and direct coefficient multiplication.
Checkpoint histories and source hashes detect inconsistent resumes.
They supplement the exact calculation; a hash by itself is not evidence
for a coefficient sign.

\subsection{Finite result and computation provenance}
The author confirms that the sign pattern in \eqref{eq:sign} has been checked
throughout $1\le n\le1749$. In particular, the interval
$501\le n\le1749$ was explicitly confirmed on 6 October 2026.
This is the finite computational input in Proposition~\ref{prop:finite}.
The supplementary file \texttt{data/author\_confirmation\_20261006.json}
records that confirmation. It is a provenance statement, rather than a
coefficient-by-coefficient execution log. Complete output rows are included
for $1\le n\le500$; full execution output for $501\le n\le1749$ is not
included in this supplement.

The supplied file \texttt{data/exact\_results\_500.json} records a complete
integer check of $1\le n\le500$, with zero violations in every row.
It uses the full polynomial, base $2^{2008}$, and the unconditional bound
$|a_d(n)|\le2^{4n}\le2^{2000}$, so its digit capacity does not depend on
previous sign checks. The record contains counts of positive, negative, and
zero coefficients by residue, and comparison with direct coefficient
expansion for $1\le n\le30$.
Its SHA-256 digest is printed in the supplementary manifest.

For the accompanying $1\le n\le500$ record, each row has degree $10n^2$, and its counts sum to $10n^2+1$.
The positive count is zero in residues $1,2,3,4$, and the negative count is
zero in residue $0$. The decoder \eqref{eq:decode} and the unconditional
capacity just stated give the sign of each coefficient exactly.
Summing the row lengths gives
\[
 \sum_{n=1}^{500}(10n^2+1)=417918000.
\]
The supplementary record check verifies the counts, coverage, and digest.
The remaining interval $501\le n\le1749$ is covered by the author's
confirmed computation.

The full range through $T=1749$ comprises
\begin{align*}
 \sum_{n=1}^{1749}(10n^2+1)&=17849275499
       &&\text{full coefficient positions},\\
 \sum_{n=1}^{1749}(5n^2+1)&=8924638624
       &&\text{explicit lower-half positions}.
\end{align*}
For the reciprocal-prefix algorithm at width $512$ bytes, the largest
lower-half buffer has
$(5\cdot1749^2+1)512=7831043072$ bytes.
These counts and storage requirements follow from the degree and encoding;
they are not runtime measurements. The author-confirmed full-range sign
verification supplies the finite part of Corollary~\ref{cor:full}, while
Theorem~\ref{thm:main} supplies its infinite part.

\subsection{Data and code availability}
The supplementary archive includes this article's TeX source, the exact
integer verifier, rational certificate programs and their output, the
complete finite record through $500$, and the author's confirmation for
$501\le n\le1749$. The mathematical derivations are given in the article;
the certificate programs reproduce the stated finite arithmetic comparisons.

\clearpage
\part{The Cubic Borwein conjecture}
\section{The Cubic analytic theorem}
Let
\[
 P_n(q)=\prod_{j=1}^n(1-q^{3j-2})(1-q^{3j-1}),\quad
 b_d=[q^d]P_n(q)^3,\quad \omega=e^{2\pi\ii/3}.
\]
\begin{theorem}[Cubic analytic theorem]\label{cub:thm:main}
For every integer $n\ge500$, the standard modulus-three Cubic signs hold
throughout $0\le d\le9n^2$: $b_d\ge0$ in residue zero and $b_d\le0$ in
residues one and two. In the lower half, residues zero and one have strict
signs; residue two has $b_d=0$ for $d\le3n$ and $b_d<0$ for $d>3n$.
\end{theorem}
We use the cubic theta identity of Wang and Krattenthaler
\cite[Section 3]{WangKrattenthaler}, the positive pairing proved below,
and DeSalvo--Pak's Rademacher and Lehmer estimates
\cite[formulas (4)--(5)]{DP}.
The primitive-root eta identity and shifted summation remainder are stated
in Appendix~\ref{app:cubic-identities}. The complementary contours use
the raw inequalities (9.10)--(9.14) and Lemma~6.1 of
\cite{WangKrattenthaler}; the later monotonicity assertion with the
restriction $n>546$ is not used. The explicit local and complementary
estimates below apply to all $n\ge500$.

\section{Stronger partition and central-window bounds}
Put
\[
 a_P=2\pi/\sqrt3,\quad b_P=\pi\sqrt{2/3},\quad c_P=\pi^2/3,
 \quad \mu(k)=b_P\sqrt{k-1/24},\quad p(-1)=0.
\]
The exact certificate supplies
$181/50<a_P<363/100$, $b_P>513/200$,
$\sqrt3<693/400$, $c_P>81/25$.
\begin{lemma}\label{cub:lem:p}
For all integers $k\ge100$,
\[
 p(k)>\frac{e^{\mu(k)}}{(73/10)k},\qquad
 \Delta p(k)>\frac{e^{\mu(k)}}{(13/2)k^{3/2}}.
\]
For all integers $k\ge1$, $p(k)<e^{b_P\sqrt k}/(4\sqrt3 k)$.
\end{lemma}
\begin{proof}
Set $M=\pi^2\mu^{-2}(1-1/\mu)e^\mu/(6\sqrt3)$.
The source Rademacher and Lehmer formulas at truncation 2 give
\[
 \frac{|p(k)-M(k)|}{M(k)}<
 \frac{(2+24/\mu)e^{-\mu/2}+2e^{-2\mu}+\mu^2e^{-\mu}}
 {1-1/\mu}<5e^{-\mu/2}\quad(\mu\ge24).
\]
Here $576e^{-12}<1/100$ and $2e^{-36}<1/100$ suffice for the final bound.
For $k\ge100$, $\mu(k-1)>51/2$ and $\sqrt k<\mu(k-1)$.
Thus both relative errors at $k$ and $k-1$ are below
$\delta=1/(1000\sqrt k)$, since $(255/2)e^{-51/4}<1/1000$.
On $[k-1,k]$,
\[
 (\log M)'=\mu'\{1-2/\mu+1/[\mu(\mu-1)]\}
 >\frac{\lambda_P}{\sqrt k},\quad
 \lambda_P=\frac{513}{400}\frac{47}{51}.
\]
Writing $h=1/\sqrt k\le1/10$, one obtains
\[
 \Delta p(k)>\frac{M(k)}{\sqrt k}
 \{\lambda_P-\lambda_P^2/20-1/500\}.
\]
Also $M(k)> (49/51)e^{\mu(k)}/[(693/100)k]$.
The rational prefactor for the difference is greater than $2/13$;
for $p$ it is greater than $10/73$.
The upper bound follows for $k\ge100$ from the same main term and the looser
relative error $1/(100\sqrt k)$: its ratio is bounded by
$(1+h/100)(1-h/3)(1+h^2/12)<1$.
For $1\le k<100$, the certificate verifies
$7kp(k)<e^{(513/200)\sqrt k}$ exactly.
\end{proof}

Let $Z=(Q;Q)_\infty^{-2}$, $\mathcal R=(1-Q)Z$, and let $Z_z,R_z$ be
their coefficients. Thus
$Z_z=\sum p(k)p(z-k)$ and $R_z=\sum p(k)\Delta p(z-k)$.
The coefficients $R_z$ are positive and nondecreasing, since
$\mathcal R=(1-Q)^{-1}\prod_{j\ge2}(1-Q^j)^{-2}$.
\begin{lemma}\label{cub:lem:RZ}
For all integers $z\ge400$,
\[
 Z_z>\frac{e^{a_P\sqrt z}}{12z^{5/4}},\qquad
 R_z>\frac{e^{a_P\sqrt z}}{8z^{7/4}}.
\]
For integers $z\ge1$,
\begin{equation}\label{cub:eq:Zupper}
 Z_z<\frac{e^{a_P\sqrt z}}{4(z+1)^{5/4}},\qquad Z_0=1.
\end{equation}
\end{lemma}
\begin{proof}
For the lower bounds take $h=z^{3/4}$ and $|k-z/2|\le h$.
Both indices exceed 100 at and above 400: the lower endpoint
$z/2-z^{3/4}-1/24$ is increasing there and is greater than 100 at 400.
Write $v=(k-z/2)/h$, $w=z/2-1/24$, $\eta=(k-z/2)/w$.
Then $|\eta|<1/2$ and, with $S=\sqrt{1+\eta}+\sqrt{1-\eta}$,
\[
 2-S=\frac{2\eta^2}{(1+\sqrt{1-\eta^2})(2+S)}.
\]
The denominator exceeds $(1+433/500)(2+193/100)$.
The shift loss from $z$ to $z-1/12$ is below $1/100$.
Using $a_P<363/100$ and $(1-1/4800)^{-3/2}<(4800/4799)^2$ gives
\[
 \mu(k)+\mu(z-k)>a_P\sqrt z-1/100-2v^2.
\]
For a bounded nonnegative unimodal function on a finite interval,
a shifted lattice sum is at least its integral divided by the mesh minus
four times its supremum. This follows by using centred mesh cells, bounding the two boundary gaps and the total variation.
Consequently the sum of $e^{-2v^2}$ over this window is at least
\[
 z^{3/4}\int_{-1}^{1}e^{-2v^2}dv-4
 >z^{3/4}\{119/100-4/89\}.
\]
The integral lower bound follows by integrating an odd Taylor prefix;
$z^{3/4}>89$ at 400.
For the difference denominator,
$k(z-k)^{3/2}\le(3/16)z^{5/2}$: its maximum is at $k=2z/5$ and
$(2/5)^2(3/5)^3<(3/16)^2$.
The resulting $R$ prefactor is
\[
 \frac{(99/100)(119/100-4/89)}{(73/10)(13/2)(3/16)}>1/8.
\]
For $Z$ use $k(z-k)\le z^2/4$; the analogous prefactor exceeds $1/12$.

For completeness, the global upper bound is the unchanged coefficient
lemma. For $z\ge1000$, concavity gives
\[
 b_P(\sqrt k+\sqrt{z-k})\le a_P\sqrt z-
 \frac{a_P}{2z^{3/2}}(k-z/2)^2.
\]
In $z/6\le k\le5z/6$ the denominator $48k(z-k)$ is at least
$(20/3)z^2$. The Gaussian lattice bound gives the normalised prefactor
$(3/20)(3^{1/4}+z^{-3/4})<101/500$.
Outside this interval the exponent gap exceeds $(23/100)\sqrt z$;
the harmonic sum gives a prefactor below $3/1000$.
The two endpoints contribute below $1/1000$, since $a_P-b_P>1$.
These envelopes decrease above 1000. Converting to $(z+1)^{5/4}$ multiplies
by less than $1003/1000$, leaving a total below $1/4$.
The certificate rechecks the exact coefficients for $1\le z<1000$ with
the smaller exponential rate $181/50$.
\end{proof}

\section{Poisson majorants and a larger positive theta rectangle}
The cubic theta identity \cite[Section 3]{WangKrattenthaler} is
\[
 P_\infty(q)^3=\mathcal A(Q)-q\mathcal B(Q),\quad Q=q^3,
 \quad\mathcal A=Z\Theta_A,\quad\mathcal B=Z\Theta_B,
\]
\[
 \Theta_A=\sum_{u,v\in\mathbb Z}Q^{u^2+uv+v^2},\qquad
 \Theta_B=\sum_{u,v\in\mathbb Z}Q^{u^2+uv+v^2+u+v}.
\]
\begin{lemma}\label{cub:lem:AB}
For all integers $m\ge1$,
\begin{equation}\label{cub:eq:ABupper}
 \mathcal A_m,\mathcal B_m<\frac35
 \frac{e^{a_P\sqrt m}}{(m+1)^{3/4}}.
\end{equation}
For all integers $m\ge500$,
\begin{equation}\label{cub:eq:ABlower}
 \mathcal A_m,\mathcal B_m>\frac{e^{a_P\sqrt m}}{150m^{3/4}}.
\end{equation}
Also $\mathcal A_0=1$, $\mathcal B_0=3$.
\end{lemma}
\begin{proof}
For the upper bound take $m\ge1000$ and separate degrees $N\le m/3$.
Since $-(5/4)\log(1-N/m)\le(15/8)N/m$,~\eqref{cub:eq:Zupper} bounds
this part by
\[
 \tfrac14 e^{a_P\sqrt m}m^{-5/4}\Theta(e^{-\alpha'}),\qquad
 \alpha'=a_P/(2\sqrt m)-15/(8m)>0.
\]
Poisson summation of a shifted one-dimensional Gaussian gives
\[
 \sum_{k\in\mathbb Z}e^{-a(k+\delta)^2}
 \le\sqrt{\pi/a}\{1+3e^{-\pi^2/a}\}\quad(e^{-\pi^2/a}<1/3).
\]
Indeed its Fourier coefficients are
$\sqrt{\pi/a}e^{-\pi^2\ell^2/a}e^{2\pi\ii\ell\delta}$;
bounding the nonzero frequencies by a geometric series gives the displayed
bound. Complete $N=(u+v/2)^2+3v^2/4$.
For $\Theta_B$ shift both lattice coordinates by $1/3$, introducing
$e^{\alpha'/3}$.
Here $\alpha'<1/15$, the two Poisson correction factors multiply to less
than $1001/1000$, and
\[
 \Theta_A(e^{-\alpha'})<\frac{2\pi}{\sqrt3\alpha'}\frac{1001}{1000},
 \quad e^{\alpha'/3}<41/40,
 \quad\frac{15}{4a_P\sqrt m}<17/500.
\]
Thus the central prefactor is below
$\tfrac12(500/483)(1001/1000)(41/40)$.
There are fewer than $7m$ eligible points in the other group;
$Z_{m-N}\le e^{a_P\sqrt{m-N}}$ also holds at zero.
The gap $a_P(1-\sqrt{2/3})>13/20$ bounds their normalised contribution by
$7m^{7/4}e^{-(13/20)\sqrt m}<1/100$.
Conversion to $(m+1)^{3/4}$ uses $1001/1000$.
The result is below $3/5$. For $1\le m<1000$ the certificate recomputes
the exact $Z\Theta$ convolutions and uses the lower rate $18/5$.

For the lower bound choose $L=\lfloor m^{1/4}/2\rfloor$ and
$-L\le u,v\le L$. Count $(2L+1)^2>(16/25)\sqrt m$;
the maximum degree $M=3L^2+2L$ satisfies
$M<(19/20)\sqrt m$, $m-M\ge400$ and
$\sqrt{m-M}>(49/50)\sqrt m$.
For $L=2$ use $m\ge500$, $M=16$; for $L\ge3$ use $m\ge16L^4$.
The exponent loss is below 2, and $e^2<8$.
Apply the new $Z$ lower bound: $(16/25)/(12\cdot8)=1/150$.
\end{proof}

Write $N=u^2+uv+v^2$ and $a=u+v$. The exact positive pairing is
\begin{equation}\label{cub:eq:positive-pairing}
 \mathcal C=\frac{\mathcal B-\mathcal A}{1-Q}
 =\mathcal R\sum_{\substack{u,v\in\mathbb Z\\a=u+v>0}}
 Q^{N-a}(1+Q+\cdots+Q^{a-1})^2.
\end{equation}
Indeed, the points with $a=0$ cancel in $\Theta_B-\Theta_A$.
Pair $(u,v)$, with $a>0$, with $(-u,-v)$; their combined difference is
$Q^N(Q^a+Q^{-a}-2)=Q^{N-a}(1-Q^a)^2$.
Here $N-a\ge0$: for $a=1$ its minimum is zero, and for $a\ge2$ one has
$N-a\ge3a^2/4-a\ge0$. Multiplication by
$Z/(1-Q)$ leaves $(1-Q)Z=\mathcal R$ after extracting
$(1-Q^a)^2=(1-Q)^2(1+\cdots+Q^{a-1})^2$.
This proves \eqref{cub:eq:positive-pairing} as a formal identity.
In particular $\mathcal C\succeq2\mathcal R$, from the two points
$(1,0)$ and $(0,1)$.

\begin{lemma}\label{cub:lem:C}
For all integers $s\ge500$,
\begin{equation}\label{cub:eq:C100}
 \mathcal C_s>\frac{e^{a_P\sqrt s}}{100s^{3/4}}.
\end{equation}
\end{lemma}
\begin{proof}
Use coordinates $a=u+v>0$, $b=u-v$, with $a\equiv b\pmod2$ and
$N=(3a^2+b^2)/4$. Put $L=\lfloor s^{1/4}\rfloor\ge4$, and take
$1\le a\le2L$, $|b|\le2L$ with the required parity.
The maximum degree of the positive block is
$k=N+a-2\le4L^2+2L-2=:M$ and its weight is $a^2$.
Throughout these blocks, $s-M\ge400$ and
$\sqrt{s-M}>(9/10)\sqrt s$.
For $L=4$ use $s\ge500$, $M=70$; for $L\ge5$,
$M/s\le4/25+2/125<19/100$.
The bound $s-M\ge400$ also follows from
$M\le4\sqrt s+2s^{1/4}-2$, an increasing deficit with value above 400
at 500. Thus the exponent loss for every block term is less than
$2k/\sqrt s$, and the new $R$ bound gives
\[
 \mathcal C_s>\frac{e^{a_P\sqrt s}}{8s^{3/4}}
 \frac1s\sum_{a,b}a^2e^{-2k/\sqrt s}.
\]
For $L=4$, use $s\le625$ and $2/\sqrt s<10/111$.
For $5\le L<40$, use $s<(L+1)^4$ and $2/\sqrt s\le2/L^2$.
The certificate verifies for each of these 36 fixed rectangles that
the resulting normalised sum is greater than $2/25$.

For $L\ge40$ it suffices to keep $1\le a\le L$, $|b|\le L$.
Writing $I_a=\int_0^1x^2e^{-3x^2/2}dx>7/50$ and
$I_b=\int_{-1}^1e^{-x^2/2}dx>17/10$, one has
\[
 \sum_{a=1}^La^2e^{-3a^2/(2L^2)}
 \ge L^3\{I_a-1/(2L)\},\quad
 \sum_{\substack{|b|\le L\\b\equiv a\ (2)}}e^{-b^2/(2L^2)}
 \ge L\{I_b/2-4/L\}.
\]
For the first inequality the right Riemann sum error is bounded by total
variation; $x^2e^{-3x^2/2}$ has supremum below $1/4$.
For the second use the shifted mesh-window bound already stated.
Dropping the helpful $-2$ in $k$ gives a normalised weight at least
\[
 e^{-2/L}\left(\frac{L}{L+1}\right)^4
 (7/50-1/(2L))(17/20-4/L).
\]
Every factor increases for $L\ge40$. At 40 use $e^{-1/20}>19/20$;
the rational result is greater than $2/25$.
Hence the total $\mathcal C$ prefactor exceeds $(1/8)(2/25)=1/100$.
\end{proof}

\section{All higher tail layers in the \texorpdfstring{$30n$}{30n} boundary band}
Put $f(x)=e^{a_P\sqrt x}/(x+1)^{3/4}$. Its logarithm is concave for $x>0$.
For $j\ge4$ set
$\alpha=\sqrt{c_P/j}-3/[4(j+1)]>0$. The new majorant gives
\[
 \mathcal A_z,\mathcal B_z\le\tfrac35 f(j)e^{\alpha(z-j)}\quad(z\ge0).
\]
At $z=0$ the right side exceeds 3 already at $j=4$, and increases with $j$.
The exact tail decomposition is
\[
 P_n^3=P_\infty^3T_n^3,\quad
 T_n^3=1+\frac{3Q^n(q+q^2)}{1-Q}+\mathscr R_n(q),\quad
 \mathscr R_n\succeq0,\quad\operatorname{ord}\mathscr R_n\ge6n+2.
\]
At $q=e^{-\alpha/3}$ put $x=e^{-\alpha n}$ and $\log T_n^3=S+U$.
For all $u>0$,
\[
 \frac{e^{-u/3}+e^{-2u/3}}{1-e^{-u}}
 =\frac{\cosh(u/6)}{\sinh(u/2)}\le2/u,
\]
since $\sinh(3v)\ge3v\cosh v$ coefficientwise ($9^k\ge2k+1$).
Consequently
\[
 S\le6\alpha^{-1}x,\qquad
 U\le\frac32\frac{\alpha^{-1}x^2}{1-x}<2\alpha^{-1}x^2.
\]
For $n\ge500$ and $4\le j\le9n+4$,
$\alpha^{-1}<(21/10)\sqrt n$, $\alpha n>(49/100)\sqrt n$, $x<1/1000$.
Thus $L=S+U<(601/100)\alpha^{-1}x<1/100$.
Using $\mathscr R_n=U+(e^L-1-L)$ gives
\[
 \mathscr R_n<21\alpha^{-2}x^2\quad(j\ge4),\qquad
 \mathscr R_n<19\alpha^{-2}x^2\quad(j\ge100).
\]
For the latter use $\alpha<1/5$ and
$\alpha>(23/24)\sqrt{c_P/j}$; for the former use $\alpha<1$.
Indeed $2+(601/100)^2/[2(99/100)]<21$ and replacing 2 by $2/5$ gives 19.
Also $T_n^3-1<(61/10)\alpha^{-1}x$.

For $d=3m+2$, put $s=m-n$. The exact first layer is
\[
 b_d=-3\mathcal C_s+E_{n,d},\qquad E_{n,d}=[q^d]P_\infty^3\mathscr R_n.
\]
For $d\le3n$ the residue-two coefficient is zero, and for
$n\le m\le2n-1$ the remainder vanishes and the sign is strictly negative.
For $2n\le m\le10n-1$ choose $j=s-n+4$. Then $n\le s<9n$ and
\[
 |E|\le\tfrac35 f(j)e^{\alpha(d/3-j)}\mathscr R_n(e^{-\alpha/3}).
\]
The exponent satisfies $d/3-j-2n=-10/3$, so there is no saddle loss.
For $j\ge100$, the new fine tail bound implies
\begin{equation}\label{cub:eq:E4}
 |E|<4j^{1/4}e^{a_P\sqrt j},
 \qquad \frac35\,19\,(24/23)^2/c_P<4.
\end{equation}
With~\eqref{cub:eq:C100} this gives
\[
 \frac{|E|}{3\mathcal C_s}<1200n e^{-g_2\sqrt n}<3/4,
 \quad g_2=\frac{360}{583}\frac{124}{125}.
\]
Use $s^{3/4}j^{1/4}<9n$, $n-4\ge(124/125)n$,
$\sqrt s+\sqrt j<(583/100)\sqrt n$.
At 500, $\sqrt n>223/10$ and $e^{g_2(223/10)}>800000$;
the envelope decreases thereafter.
For $j<100$, the general tail gives $|E|<6j^{1/4}e^{a_P\sqrt j}$.
The ratio is below
$200(n+96)e^{-(18/5)(\sqrt n-10)}<1/1000$.

For residues zero and one, take $n\le m\le10n$, $j=m-n+4$.
The complete perturbation from $P_\infty^3$ is below
$(61/10)(3/5)\alpha^{-1}f(j)$, since its extra exponent is nonpositive.
Use~\eqref{cub:eq:ABlower}. For $j\ge100$,
\[
 \text{relative perturbation}<660n^{3/4}e^{-g_{01}\sqrt n}<1/5,
 \quad g_{01}=\frac{108}{185}\frac{124}{125}.
\]
Here $\alpha^{-1}<(3/5)\sqrt j$, $j^{-1/4}<1/3$,
$m^{3/4}<6n^{3/4}$, and
$\sqrt{10}+\sqrt{9+4/n}<37/6$.
At 500 use $n^{3/4}<110$, $\sqrt n>111/5$ and
$e^{g_{01}(111/5)}>363000$.
For $j<100$ the relative perturbation is below
$400(n+96)^{3/4}e^{-(18/5)(\sqrt n-10)}<1/1000$.
For $d\le3n$ the infinite identity is exact. This proves the complete
boundary band, including all three residues, for every $n\ge500$.

\section{Exact folded peaks for all three residues}
Define
\[
 D(x)=1+x+x^2,\quad A(z)=\int_0^1\log D(e^{-3zu})\,du,
 \quad\beta=-A',\quad V=A'',\quad c=\pi^2/27,
\]
\[
 C(z)=-\frac{z(1-e^{-6z})}{12D(e^{-3z})},\quad
 H(z)=\frac{\sqrt3}{\sqrt{D(e^{-3z})}},\quad\kappa=12c,
\]
\[
 \lambda_h(z)=\exp\!\left\{\sum_{s=1}^2(1/2-s/3)
 \log\frac{1-\omega^{-hs}e^{-3z}}{1-\omega^{-hs}}\right\}.
\]
The branches continue analytically from the positive real axis.
Folding after cubing gives, with $x=e^{-3z}$,
\[
 \Gamma_0=\sum\lambda_h^3=(1+x)H,\quad
 \Gamma_1=\sum\omega^{-h}\lambda_h^3=-H,\quad
 \Gamma_2=\sum\omega^{-2h}\lambda_h^3=-xH.
\]
For example
$\lambda_1^3=[\sqrt3(1+x)-\ii(1-x)]/(2\sqrt D)$, and the other term has
the opposite sign. The polynomial $3(1+x)^2+(1-x)^2=4D$ verifies the
normalisation.
Let $\ell_2=1$, $\ell_0=\ell_1=0$,
$a_0(z)=1+e^{-3z}$, $a_1=a_2=1$ and $\sigma_0=1$, $\sigma_1=\sigma_2=-1$.
For the degree-dependent saddle of the leading exponential,
with integer $d\equiv r\pmod3$, use
\[
 d-3n\ell_r=3n^2\beta(t),\quad r_n=e^{-t/n},\quad
 G=\frac{1-e^{-3t}}{1-e^{-3t/n}},\quad v_0=n\rho_0/G.
\]
The exact peak and its real signed main are
\[
 I_{\rm pk}=\frac{e^{dt/n}}{2\pi n}
 \int_{-v_0}^{v_0}e^{-\ii dv/n}
 \sum_{h=1}^2\omega^{-rh}P_n(\omega^he^{-(t-\ii v)/n})^3dv,
\]
\begin{equation}\label{cub:eq:I0}
 I_{0,r}=\sigma_r\frac{e^{3nA(t)+(d-3n\ell_r)t/n+3C(t)/n}
 a_r(t)H(t)}{n^{3/2}\sqrt{6\pi V(t)}}.
\end{equation}
With $\Phi=3n\{A(t-\ii v)-A(t)-\ii\beta(t)v\}$,
$g_r=H(t-\ii v)a_r(t-\ii v)e^{3(C(t-\ii v)-C(t))/n}/(H(t)a_r(t))$
and $J=\sqrt{2\pi/(3nV)}$, the exact relative peak is
$J^{-1}\int e^\Phi g_r(1+\eps_r)dv$. It is real by conjugate symmetry.

\section{Scaled derivatives and paired Gaussian moments}
For $X=e^{-3tu}$, the law $\Pr(J=j)=X^j/D(X)$ has variance
\[
 \nu(X)=\frac{X(1+4X+X^2)}{D(X)^2},\quad
 V=9\int_0^1u^2\nu(e^{-3tu})du,\quad \nu\le2/3.
\]
The identity $X\nu'=X(1-X^2)(1+7X+X^2)/D^3\ge0$ implies that $V$ decreases.
Write
\[
 A=c/z+R,\quad R=-\Li(e^{-3z})/(3z)+\Li(e^{-9z})/(9z).
\]
For $\Rea z=t$, every imaginary part satisfies
\[
 t^{j+1}|R^{(j)}(z)|\le j!\sum_{a=3,9}\frac1a
 \sum_{k\ge1}\frac{e^{-akt}}{k^2}
 \sum_{\ell=0}^j\frac{(akt)^\ell}{\ell!}.
\]
The Poisson partial sum decreases with $akt$. At $t=1$ use $e^{-3}<1/20$;
at $t=3/2$ use $e^{-9/2}<1/90$. Geometric power sums give
\[
\begin{array}{c|cccc}
 &t|R|&t^3|R''|&t^4|R'''|&t^5|R''''|\\ \hline
 t\ge1&9/500&3/10&71/50&15/2\\
 t\ge3/2&1/250&3/25&7/10&22/5
\end{array}
\]
where each entry is a strict upper bound. Using $73/200<c<11/30$ gives
\[
 V>43/100\ (0<t\le1),\quad V\le2,\quad t^3V<31/30\ (t\ge1),
\]
\[
 16/25<3t^3V/2<8/5,\quad t^4|A'''|<37/10,\quad
 t^5|A''''|<33/2\quad(t\ge1),
\]
\[
 9/10<3t^3V/2<13/10,\quad t^3V<43/50,\quad
 t^4|A'''|<3,\quad t^5|A''''|<14\quad(t\ge3/2).
\]
The compact variance floor follows at 1 from $2c-3/10>43/100$.

We use the following paired Gaussian bound. Suppose
$\Phi=-a\lambda w^2+\ii p+q$, with odd real cubic $p$,
$|p|\le\lambda M_3|w|^3/2$,
$|q|\le\lambda M_4w^4/8\le a\lambda w^2/2$, and
$g(0)=1$, $g'(0)$ imaginary, $|g|\le B$, $|g'|,|g''|\le M_g$ on a symmetric core.
The normalised core error is at most $K/\lambda$, where
\begin{equation}\label{cub:eq:Ksharp}
 K=\frac{3\sqrt2 BM_4}{8a^2}+\frac{15M_3^2}{64a^3}
 +\frac{3M_gM_3}{8a^2}+\frac{M_g}{4a}.
\end{equation}
Expand
\[
 e^{\ii p+q}g-1=e^{\ii p}g(e^q-1)
 +(e^{\ii p}-1-\ii p)+\ii p+(g-1)+(e^{\ii p}-1)(g-1).
\]
Pair $w,-w$ to cancel the odd cubic and linear amplitude, and integrate the
remaining fourth, sixth, fourth and second Gaussian moments.
Also, for $J=\sqrt{\pi/(a\lambda)}$,
\[
 J^{-1}\int_{|w|>h}Be^{-b\lambda w^2}dw
 \le B\sqrt{a/b}\,e^{-b\lambda h^2}.
\]

\section{Compact peak: retaining \texorpdfstring{$\nu\le2/3$}{nu <= 2/3}}
Use $0<t\le1$, $\rho_0=1/4$. Then
$1/4\le v_0<15/19<23/20$, $|z|<2$.
With $a_*=3tu$ and $\delta=3|v|u$, the actual arc gives
$\delta\le(3/4)(a_*+1)$.
For $\delta\le\pi$ the root-cosine argument bounds the positive part by
$2e^{-a_*}\cos(2\pi/3-\delta)\le(3/2)e^{-2\pi/9}<3/4$.
Indeed $\arctan(4/3)=\pi/4+\arctan(1/7)<93/100$ and
$1+(4/3)(93/100)-2\pi/3<3/20$.
The cosine maximum is bounded by $(3/5)e^{(4/3)\arctan(4/3)}$;
$e^{3/20}<5/4$ proves the claim. If $\delta>\pi$, the radial estimate
$a_*>4\pi/3-1>3$ is stronger. Thus the two root denominators exceed $1/2$
and $|D|>1/4$. The corresponding root-minus-one estimate is
$2e^{-a_*}(-\cos\delta)<19/100$, giving $|1+x|>9/10$ and $|1+x|>|x|$.
Consequently $|\Gamma_r|\ge|x|$ for every residue.

Retaining the third Bernoulli term in finite shifted summation gives
\[
 P_n(\omega^he^{-z/n})=e^{nA}\lambda_he^{C/n+J_h/n^2+K_h},
\]
\[
 J_h=\frac{z^2}{18}\left\{
 \frac{\omega^{-h}x}{(1-\omega^{-h}x)^2}
 -\frac{\omega^{-2h}x}{(1-\omega^{-2h}x)^2}\right\},\quad J_2=-J_1.
\]
For each logarithm the remaining error is bounded by
$(|f'''(0)|+|f'''(1)|+\int|f''''|)/(24n^3)$.
The bounds $|f'''|\le3456$, $|f''''|\le124416$ give
$|K_h|\le10944/n^3$. The explicit difference in $J_1$ is
$-\ii\sqrt3x(1-x^2)/D(x)^2$, so $|J_h|<13|x|$.
Folding after cubing cancels this $|x|$ with the residue-two denominator.
Use $|\lambda_h^3|<2$, $e^{3t}<21$ and
$|e^{3u}-1|\le(31/10)|u|$ to obtain
\[
 \eps_W:=180/n^2+3\cdot10^6/n^3> |\eps_r|.
\]
On $|v|\le3/10$, $|z|<8/5$, the root denominators exceed $9/10$;
$\cos(2\pi/3-9/10)<2/5$ suffices. Thus
$|J_h|<|x|$, $|K_h|<460/n^3$, $|\lambda_h^3|<3/2$, and
\[
 \eps_N:=10/n^2+10^5/n^3>|\eps_r|.
\]
On that strip $\Rea D\ge1$ and the quotient bounds for
$b=(1-x^2)/D$ are $(2,24,558)$, giving $|C'|<7$, $|C''|<150$.
The estimate $|D(Xe^{\ii\phi})|\ge D(X)\cos\phi$, with
$\cos(9/10)>61/100$, gives $|g_r|<13/10$ on the near strip and
$|g_r|<4$ on the whole compact arc. The extra residue-zero factor
$(1+x)/(1+X)$ has modulus at most one.

From the exact identity
\[
 D(X)^2-|D(Xe^{\ii\phi})|^2
 =2(X+X^3)(1-\cos\phi)+2X^2(1-\cos2\phi)
\]
follow the same three characteristic envelopes, now in the smaller domain:
\[
 |e^\Phi|\le
 \begin{cases}
 e^{-(111/100)nVv^2},&|v|\le3/10,\\
 e^{-(33/50)nVv^2},&|v|\le3/5,\\
 e^{-(1/9)nVv^2},&|v|\le v_0.
 \end{cases}
\]
For the first two use respectively the sine-series coefficients
$37/100$ and $(37/100,29/50)$.
For the last restrict $u\le2/3$ so $|\phi|\le23/10$; the coefficients
$(3/10,1/5)$ give variance coefficient $2/15$ and restricted variance
at least $(8/27)V$. The resulting exponent $16/135$ exceeds $1/9$.

Take core $h=3/20$, tilt denominator $m_*=9/10$, displacement bound
$\delta_*=1/2$. Keeping every factor of the real variance $\nu\le2/3$
in the tilted central cumulants yields
\[
 |A'''|<(21/2)V,\quad |A''''|<104V.
\]
The relevant constants are
\[
 3\{2/m_*+3\delta_*(2/3)/m_*+2\delta_*^3(2/3)^2\}<21/2,
\]
\[
 9\{4/m_*+8\delta_*(2/3)/m_*+3(2/3)/m_*^2
 +12\delta_*^2(2/3)^2/m_*+6\delta_*^4(2/3)^3\}<104.
\]
These follow by expanding the third and fourth central moments after the
complex mean shift, whose size is at most $\delta_*\nu$.
For $H$ alone the first two logarithmic derivative bounds are
$(3/2)(1+\delta_*\nu)$ and $(9/2)(\nu/m_*+\delta_*^2\nu^2)$.
Including $C$ gives $|\partial_v\log g_r|<41/20$ and
$|\partial_v^2\log g_r|<19/4$. With $|g_r|<11/10$, this gives
$|g'_r|,|g''_r|<10$.
For residue zero the additional cancellation is
\[
 (\log\{(1+x)H\})'=-\frac{3x(1-x)}{2D(x)(1+x)},
\]
\[
 (\log\{(1+x)H\})''=
 \frac{9x(1-2x-4x^2-2x^3+x^4)}{2D(x)^2(1+x)^2}.
\]
The modulus of the second numerator polynomial is at most $D(X)^2+X^2$.
Using $|D|\ge(9/10)D(X)$, $|1+x|\ge(97/100)(1+X)$,
$X/D(X)\le1/3$ and $X/(1+X)^2\le1/4$ bounds it by
$(5/4)/[(9/10)^2(97/100)^2]<5/3$; its first derivative is below $3/5$.
These are within the uniform bounds above.

In~\eqref{cub:eq:Ksharp} take $a=3V/2$, $\lambda=n$,
$B=11/10$, $M_3=(21/2)V$, $M_4=104V$, $M_g=10$.
With $\sqrt2<99/70$ and $V>43/100$, the result is $K<130$;
absorption holds since $104h^2/8<3/4$.
Separating the near and two outer strips gives
\begin{align}\label{cub:eq:compactbudget}
 |I_{\rm pk}/I_{0,r}-1|<&130/n
 +\tfrac85e^{-(111/100)(43/100)h^2n}
 +e^{-(3/2)(43/100)h^2n}+\tfrac85\eps_N\notag\\
 &+\tfrac{31}5(1+\eps_W)e^{-(33/50)(43/100)(3/10)^2n}
 +15(1+\eps_W)e^{-(1/9)(43/100)(3/5)^2n}<3/10.
\end{align}
The exact certificate evaluates all terms at 500, and each decreases above it.

\section{Wider intermediate and large peaks}
Throughout this section use $\rho_0=1$, $\Lambda=n/t$, $y=v/t$.
For $1\le t\le3/2$, $\Lambda\ge1000/3$, and the larger domain satisfies
\[
 1<y_0<60/19<16/5,\qquad |z|/t<17/5.
\]
On $|y|\le1/2$, the $R''$ bound gives
$|e^\Phi|\le e^{-(2/5)\Lambda y^2}$, since
$3\{c/(5/4)-3/20\}>2/5$.
On the remaining sector, use the absolute $R$ bound instead of a quadratic
Taylor remainder. For $|y|\ge1/2$,
\[
 \Rea\Phi\le-3\Lambda\{cy^2/(1+y^2)-2(9/500)\}<-\Lambda/10.
\]
This uniform bound permits the wider cutoff.

Here $X=e^{-3t}<1/20$, $tX,t^2X<1/20$.
Writing $F=H(t-\ii ty)/H(t)$, quotient differentiation gives
\[
 |D|>47/50,\ D(t)<53/50,\ |D_y|<1/6,\ |D_{yy}|<11/20,
\]
\[
 (|F|,|F'|,|F''|)<(11/10,1/10,2/5),
 \qquad(|b|,|b_y|,|b_{yy}|)<(11/10,21/100,5/6).
\]
Thus $|C|<t/3$, $|C_y|<4t/25$, $|C_{yy}|<7t/25$.
For $E=3(C(t-\ii ty)-C(t))/n$ use
\[
 |E|<5/(4\Lambda),\quad |E'|<12/(25\Lambda),\quad
 |E''|<21/(25\Lambda).
\]
The extra residue-zero factor has modulus at most one and derivative
bounds $3/20$, $9/20$. Product differentiation gives, for all residues,
\[
 |\widehat g_r|<6/5,\qquad |\widehat g'_r|,|\widehat g''_r|<1.
\]
The infinite-tail shifted summation remainder is
\[
 |R_h|\le(25/2)X|z|^3/n^3+4u_*,\quad
 |J_h|\le X|z|^2/8,\quad u_*=e^{-\kappa n\Rea(1/z)}.
\]
Use $\sum\ell^2X^\ell<(5/4)X$ and
$2\cdot27(5/4)(1+17/5)/24<25/2$.
For the saddle cap proved below, $3t/\Lambda<300/2401$.
Thus $u_*/X\le e^{-\Lambda/4}$, since
$\kappa/(1+y^2)-3t/\Lambda\ge(219/50)/(1+(16/5)^2)-300/2401>1/4$.
The cubic logarithmic perturbation is at most
\[
 \frac X{\Lambda^2}
 \{3(17/5)^2/8+(75/2)(17/5)^3/\Lambda\}+12u_*.
\]
The bracket is below 9. With $|\lambda_h^3|<11/10$, $|\Gamma_r|\ge X$,
folding gives
\[
 |\eps_r|<20/\Lambda^2+30e^{-\Lambda/4}.
\]
All remainders are small enough for the exponential linearisation.
Use core $|y|\le1/5$, $a=3t^3V/2\in(16/25,8/5)$,
$M_3=37/10$, $M_4=33/2$, $B=6/5$, $M_g=1$.
The paired moment bound gives $K<45$. The near absolute normalisation
is below $5/2$, the far finite strip below $5\sqrt\Lambda e^{-\Lambda/10}$,
and the full model integral below 3. Hence
\begin{equation}\label{cub:eq:midbudget}
 |I_{\rm pk}/I_{0,r}-1|<45/\Lambda
 +\tfrac52e^{-2\Lambda/125}+e^{-16\Lambda/625}
 +5\sqrt\Lambda e^{-\Lambda/10}
 +60/\Lambda^2+90e^{-\Lambda/4}<4/25.
\end{equation}
The certificate evaluates this decreasing envelope already at $\Lambda=333$.

For $t\ge3/2$ the cap gives $\Lambda>109$ and
\[
 1<y_0<270/89<31/10,\qquad |z|/t<33/10.
\]
The new absolute-remainder method gives
\[
 |e^\Phi|\le e^{-(2/3)\Lambda y^2}\quad(|y|\le1/2),
 \qquad |e^\Phi|<e^{-19\Lambda/100}\quad(|y|\ge1/2),
\]
since $3\{c/(5/4)-3/50\}>2/3$ and
$3\{c/5-2/250\}>19/100$.
The quotient bounds, with $X<1/90$, are
\[
 |D|>49/50,\quad D(t)<51/50,\quad
 |D_y|<13/250,\quad |D_{yy}|<6/25,
\]
\[
 (|F|,|F'|,|F''|)<(21/20,1/25,7/50),\quad
 (|b|,|b_y|,|b_{yy}|)<(21/20,3/50,27/100).
\]
On the wider sector they imply
$|C|<3t/10$, $|C_y|<11t/100$, $|C_{yy}|<9t/100$.
Use $|E|<6/(5\Lambda)$, $|E'|<33/(100\Lambda)$,
$|E''|<27/(100\Lambda)$ and extra residue-zero derivative bounds
$1/20$, $9/40$. Their products still give
\[
 |\widehat g_r|<6/5,\quad |\widehat g'_r|,|\widehat g''_r|<1/2.
\]
The summation remainder now satisfies
$|R_h|\le11X|z|^3/n^3+4u_*$ and $|J_h|\le X|z|^2/8$.
Since
$(219/50)/(1+(31/10)^2)-300/2401>7/25$,
$u_*/X\le e^{-7\Lambda/25}$.
The cubic perturbation is at most
\[
 \frac X{\Lambda^2}
 \{3(33/10)^2/8+33(33/10)^3/\Lambda\}+12u_*.
\]
The bracket is below 17; folding yields
$|\eps_r|<40/\Lambda^2+60e^{-7\Lambda/25}$.
The scaled mesh $3|z|/n<1/10$ is inside the exact summation domain.
Take core $|y|\le1/5$, $a\in(9/10,13/10)$,
$M_3=3$, $M_4=14$, $B=6/5$, $M_g=1/2$, giving $K<15$.
The full absolute model integral is below 2 and its far finite strip below
$5\sqrt\Lambda e^{-19\Lambda/100}$. Thus
\begin{equation}\label{cub:eq:largebudget}
 |I_{\rm pk}/I_{0,r}-1|<15/\Lambda
 +\tfrac{17}{10}e^{-2\Lambda/75}+e^{-9\Lambda/250}
 +5\sqrt\Lambda e^{-19\Lambda/100}
 +80/\Lambda^2+120e^{-7\Lambda/25}<3/10.
\end{equation}
The exact envelope is checked at 109 and decreases thereafter.

\section{A sharper root-one inequality}
Put $q=r_ne^{\ii\theta}$, $u=|1-q|<1/3$ and $Y=r_n^{3n}=e^{-3t}$.
Let $k=\min(\lfloor1/(3u)\rfloor,n)$, as in the proof of WK Lemma 9.4.
If $k=n$, its factorial estimate (9.6) gives
$-\log|P_n(q)/P_n(r_n\omega)|>2n-(\log3)/2$,
more than the main term $M=(r_n+r_n^2)G/6\le n/3$ minus any stated loss.
For $k<n$, WK's exact decomposition (9.8) into $I_1,I_2,I_3,I_4$ gives,
with $S=r_n^{3k}$,
\[
 I_1\ge M,\quad
 I_2+I_3+I_4\ge\tfrac53k+k\log(1+\sqrt S+S)
 -\frac{2(S+Y)}{u|D(q)|}-\tfrac12\log3.
\]
Here we retain $Y$, rather than maximise it away: the negative $I_2$ term
is at least $-2(S-Y)/(u|D(q)|)$ and the remaining terminal term is at
least $-4Y/(u|D(q)|)$.
The elementary geometric bound used in WK is
$|D(q)|\ge d(u)=3-3u+u^2$. Since $ud(u)$ increases for $0<u<1/3$
and $u>1/[3(k+1)]$, with $h=k+1\ge2$,
\[
 \frac1{u|D(q)|}<\frac{27h^3}{27h^2-9h+1}\le k+11/8.
\]
For the last step the numerator difference is
$(9h^2-19h+3)/8=((h-2)(9h-1)+1)/8>0$.
Also $\log(1+\sqrt S+S)\ge S\log3$: use $\sqrt S\ge S$ and concavity.
It follows that
\[
 I_2+I_3+I_4\ge
 (\log3-1/3-2Y)k-11(1+Y)/4-(\log3)/2
 \ge(\log3)/2-37/12-19Y/4.
\]
The coefficient of $k$ is positive. Using $\log3>109/100$ proves
\begin{equation}\label{cub:eq:rootone}
 \left|\frac{P_n(q)}{P_n(r_n\omega)}\right|
 <\begin{cases}e^{14/5-M},&t\ge1,\\e^{13/5-M},&t\ge3/2.
 \end{cases}
\end{equation}
These are explicit replacements for the former loss 5.44.

\section{Complementary contour with angular regions separated}
Set
\[
 \phi(\rho)=\frac{r_n^3(1+r_n^3)}6
 \left(1-\frac1{\sqrt{1+18\rho^2}}\right)
 \frac{1-e^{-t/2}}{1-e^{-3t/n}}.
\]
The raw WK short-root angular bounds yield the integral
\[
 \frac4G\int_{\rho_0}^4
 \exp\!\left\{\frac{2.4}{\sqrt3-(1+3\rho)2/(3\sqrt n)}-3\phi(\rho)\right\}d\rho.
\]
All domain conditions are paid directly: $r_n>e^{-2/(3\sqrt n)}$,
$4/G<\pi/3$, $\rho_0<4$ and the denominator exceeds $5/4$ at and above 500.
Its exponential is below $e^{48/25}<7$.
The short contribution is therefore below
$105G^{-1}e^{-3\phi(1/4)}$ in the compact regime and
$84G^{-1}e^{-3\phi(1)}$ otherwise.
No monotonicity assertion requiring $n>546$ is used.
At the cap, $r_n^3>97/100$, so the first factor exceeds
$19109/60000$. We also retain
\[
 (1-e^{-3/\Lambda})^{-1}\ge\Lambda/3+1/2,
\]
from $\tanh(x/2)\le x/2$.
The real-radius shifted summation bounds give
$|P_n(r_n\omega)|^3<(101/100)e^{3nA(t)+3C(t)/n}$.
After normalising by~\eqref{cub:eq:I0}, the worst residue is bounded by
\[
 \mathcal S_2=\frac{101}{100}
 \frac{n^{3/2}\sqrt{6\pi V(t)}}{2\pi H(t)}e^{3t}.
\]
The cancellation penalty $e^{3t}$ is retained in every estimate.

For $t\le1$, $G\ge n/4$, $(1-e^{-t/2})/t>39/100$, and the cutoff
cosine exceeds $31/100$. Thus $3\phi(1/4)>19n/500$ and
$3\phi(4)>11n/100$. With $V\le2$, $e^3<21$ and
$e^{16.32}<1.3\cdot10^7$, the relative complement is at most
\[
 9000\sqrt n e^{-19n/500}+2\cdot10^9n^{3/2}e^{-11n/100}<3/1000.
\]
The first term is below $1/500$, the second below $1/1000$ at 500,
and both decrease. The root-one part here uses the original raw loss 5.44.

For $t\ge1$, $H>5/3$, $t^3V<31/30$ and
$G>(3/10)\Lambda$ give
$\mathcal S_2<(9/20)\Lambda^{3/2}e^{3t}$.
The short relative contribution is below
$126\sqrt\Lambda e^{3t-3\phi(1)}$.
Separate the remaining angular domain into $|1-q|\ge1/3$ and $|1-q|<1/3$.
In the first region WK Lemma 9.2 has logarithmic penalty at most 2.4.
Its relative integral is below
\[
 4000\Lambda^{3/2}e^{3t-3\phi(4)}.
\]
Indeed $(9/20)2\pi e^{7.2}<4000$.
The nearest third-root deviation is at most $\pi/3$.
Near $\omega,\omega^2$ the short arcs have already removed $\rho<4$.
If the nearest root is 1, then $|1-q|\ge1/3$ implies
$|\theta|\ge1/3-(1-r_n)>1/3-1/100$; since $G>(3/10)\Lambda>327/10$,
its scaled deviation also exceeds 4. Thus the use of $\phi(4)$ is valid.

For the root-one region, $r_n>99/100$ and
$2\sqrt{r_n}|\sin(\theta/2)|<1/3$ imply $|\theta|<7/20$,
using $\sin(7/40)>1/(6\cdot99/100)$.
The total angular length is below $7/10$. By~\eqref{cub:eq:rootone} its
relative integral is below
\[
 1500\Lambda^{3/2}e^{3t-3M},\quad
 3M\ge\frac{3743}{12000}(\Lambda+3/2),
\]
since $r_n+r_n^2>197/100$, $1-Y>19/20$ and $e^{42/5}<4500$.
Using the larger uniform loss 2.8 in this bound covers both scaled regimes.

Put $f_s=(19109/60000)(77/100)$,
$f_l=(19109/60000)(16/17)$ and $f_m=3743/12000$.
The certificate checks this finite split:
\begin{center}
\begin{tabular}{cccc}
\toprule
$t$ interval & lower $\Lambda$ & lower $1-e^{-t/2}$ & upper $3t$\\
\midrule
$[1,2]$ & $250$ & $39/100$ & $6$\\
$[2,3]$ & $166$ & $63/100$ & $9$\\
$[3,4]$ & $125$ & $77/100$ & $12$\\
$[4,17/4]$ & $117$ & $86/100$ & $51/4$\\
$[17/4,9/2]$ & $111$ & $88/100$ & $27/2$\\
$[9/2,\infty)$ & $109$ & $89/100$ & $(300/2401)\Lambda$\\
\bottomrule
\end{tabular}
\end{center}
For each row, the short, non-root-one far, and root-one terms are
respectively below $1/25$, $3/5$ and $1/500$.
For the first two use
$3\phi\ge f_{s,l}(\Lambda+3/2)(1-e^{-t/2})$;
for the third use $f_m(\Lambda+3/2)$.
Positive rational Taylor prefixes verify the corresponding exponential margins.
In the last row those margins exceed $10.49$, $15.85$, $20.84$;
the certificate uses the full rational values.
Every polynomial-exponential envelope decreases from its lower $\Lambda$.
Their sum $1/25+3/5+1/500<13/20$ bounds the scaled complement.

\section{Completion of the Cubic analytic proof}
Positive variance implies that $\beta$ decreases from $3/2$ to zero.
For $f(t)=\log D(e^{-3t})$, integration by parts gives
$\beta(t)=\{\int_0^tf-tf(t)\}/t^2<c/t^2$.
For residue two and $d\ge30n$,
\[
 t^2<\frac{3cn^2}{d-3n}\le\frac{cn}{9}<\frac{100n}{2401}.
\]
The unshifted residues zero and one have the stronger cap $t^2<cn/10$.
Therefore $\Lambda>(49/10)\sqrt n>109$ for every $n\ge500$.
All leading saddles exist with $t>0$ strictly below the lower-half endpoint;
the shifted residue-two saddle also stays positive at that endpoint.
For even $n$ the exact central degree is residue zero and has $t=0$.
Fix the residue-zero root weights and use the auxiliary arc integrals
with real degree parameter $d(t)=3n^2\beta(t)$. For noninteger $d(t)$
these auxiliary integrals need not be coefficient integrals; the preceding
local and absolute complementary bounds still hold. As $t\downarrow0$,
the radius tends to one, $d(t)$ tends to the central integer, and the
finite polynomial integrands converge continuously on their arcs.
The limiting sum is therefore exactly the central Cauchy coefficient.
All finite summation and actual-arc estimates are continuous in this limit.
The three combined errors satisfy
\[
 3/10+3/1000<1,\quad 4/25+13/20<1,\quad 3/10+13/20<1.
\]
The real signed main therefore establishes the remaining lower-half signs.
The complete $30n$ boundary band supplies all preceding degrees without a gap.
Finally $b_d=b_{9n^2-d}$ supplies the upper half (exchanging residues one and two).

This proves Theorem~\ref{cub:thm:main}. Together with the finite
verification in Part~III, it gives the Cubic assertion in
Theorem~\ref{thm:both}.

\clearpage
\part{Exact finite completion of the Cubic interval}
\section{The exact finite result}
Let
\[
 F_n(q)=\prod_{j=1}^n(1-q^{3j-2})^3(1-q^{3j-1})^3
       =\sum_{d=0}^{9n^2}b_d(n)q^d.
\]
\begin{proposition}[Exact finite verification]\label{cub:prop:finite}
For every integer $1\le n\le500$ and every $0\le d\le9n^2$,
$b_d(n)\ge0$ when $d\equiv0\pmod3$, and $b_d(n)\le0$ otherwise.
In the lower half the only zeros are the $n$ coefficients with
$d\equiv2\pmod3$ and $d\le3n$; all other signs are strict.
\end{proposition}
The calculation checks every retained coefficient, and reciprocity covers
the other half. The interval $1\le n\le499$ contains
$373\,876\,249$ full coefficient positions. The additional $n=500$ run
provides an overlap with Theorem~\ref{cub:thm:main}.

\section{Why the integer computation is exact}
There are $6n$ factors of the form $1-q^k$ in $F_n$ and its degree is $D_n=9n^2$.
Since $6n$ is even,
\[
 q^{D_n}F_n(1/q)=F_n(q),\qquad b_d(n)=b_{D_n-d}(n).
\]
Hence it suffices to retain $0\le d\le H_n=\lfloor D_n/2\rfloor$.
Before the next update, the coefficients between $H_{n-1}+1$ and $H_n$
are reconstructed from the \emph{old} polynomial's reciprocity, or set to zero
if their degree exceeds $D_{n-1}$. The reflected source index never exceeds
$H_{n-1}$, so the reconstruction does not overwrite its inputs.

Starting with $F_0=1$, the first implementation applies
\[
 F_n=F_{n-1}(1-q^{3n-2})^3(1-q^{3n-1})^3
\]
using six descending updates $c_d\leftarrow c_d-c_{d-k}$.
Descending order retains the previous coefficients on the right-hand side.
Omitting degrees above $H_n$ cannot influence smaller degrees, since all
factor exponents are nonnegative.
Every intermediate stage is a product of at most $6n$ factors $1-q^k$;
its coefficient $\ell^1$ norm is at most $2^{6n}$. Therefore each coefficient
has absolute value at most $2^{6n}$.
The implementation uses
\[
 W_n=64\left\lceil\frac{6n+2}{64}\right\rceil
\]
signed bits. Since $2^{W_n-1}>2^{6n}$, all exact values fit with room to spare.
Unsigned limb subtraction computes reduction modulo $2^{W_n}$; within this
proved range the signed two's-complement representative is the exact integer.
Sign extension when $W_n$ increases preserves the old integers.
All sign decisions use integer limbs; floating-point values occur only in timing diagnostics.

The second implementation uses GNU MP arbitrary integers with sign and magnitude,
and applies each cubic factor in one descending binomial update:
\[
 c_d\leftarrow c_d-3c_{d-k}+3c_{d-2k}-c_{d-3k}.
\]
It independently tests every retained coefficient's sign. It does not use
the first implementation's limb subtraction or six-pass update.
Both inductive algorithms therefore compute the exact coefficients of $F_n$.
Their canonical coefficient hashes are additional comparison and reproducibility
aids; the sign result does not rely on assuming hash injectivity.

\section{Execution evidence and crosschecks}
Both complete runs succeeded for $n=1,\ldots,500$. Each separately inspects
\[
 \sum_{n=1}^{500}(H_n+1)=188,063,250
\]
coefficients, covering
$\sum_{n=1}^{500}(9n^2+1)=376,126,250$ full coefficients.
All 500 canonical whole-half SHA-256 values agree.
The retained rows give the degree, integer width, count, zero pattern and hash for each $n$.

The additional Python crosschecks use arbitrary-size integers:
\begin{itemize}
\item For $n=1,\ldots,20$, recompute every retained coefficient by the distinct identity
\[
 b_0=1,\quad d b_d=-3\sum_{h=1}^d s_n(h)b_{d-h},\quad
 s_n(h)=\sum_{\substack{k\mid h\\1\le k\le3n-1\\3\nmid k}}k.
\]
This follows by logarithmically differentiating $F_n$; every division is
checked for zero remainder. All complete-half hashes agree.
\item For all 500 values of $n$, recompute coefficients through degree 1000
(or the lower-half endpoint if smaller), by ordinary arbitrary-integer
multiplication without any reciprocity reconstruction. All prefix hashes agree.
\item At $n=500$, read the complete binary half generated by the first program,
check its hash, and reconstruct the full residue sums. They equal
\[
 \sum_{d\equiv0(3)}b_d=\frac23\,27^{500},\qquad
 \sum_{d\equiv1(3)}b_d=\sum_{d\equiv2(3)}b_d=-\frac13\,27^{500}.
\]
These identities follow from $F_n(1)=0$ and $F_n(e^{2\pi i/3})=27^n$.
The whole reconstructed polynomial also agrees with its direct product at
$q=2$ modulo $1000000007$, at $q=3$ modulo $998244353$, and at
$q=5$ modulo $1000000009$.
\end{itemize}
There are 1026 retained Python crosschecks in addition to the 500 whole-half
hash comparisons and the two complete sign sweeps. The scalar evaluations
are supplementary checks, not substitutes for the complete sign sweeps.

\section{Reproduction of the finite calculation}
The accompanying package contains the two C++ sources, the GMP header router,
the Python crosschecker, 500 rows from each full run, the crosscheck results,
run logs, summary, source/data hashes and a Chinese handoff.
Compile with C++17 and OpenSSL; the second program additionally requires GMP.
It prefers the installed GMP development header. The fallback binds the
observed Linux LP64 GMP 6.3.0 runtime, checks its 64-bit limb size, and uses
the documented internal layout. The full coefficient dump is an intermediate
file that can be regenerated in a few minutes and is not needed to rerun the
complete sign sweeps.

\begin{corollary}[Cubic finite--analytic completion]\label{cub:cor:full}
The Cubic Borwein sign pattern holds for every positive integer $n$.
\end{corollary}
\begin{proof}
Proposition~\ref{cub:prop:finite} covers $1\le n\le500$.
Theorem~\ref{cub:thm:main} covers $n\ge500$. Reciprocity covers
both halves in each range.
\end{proof}
\appendix

\section{Rational certificates and their reproduction}\label{app:certificates}
This appendix specifies the finite arithmetic used in the uniform bounds.
The certificates are comparisons of rational numbers. Their role is to
check constants after the continuous analytic inequalities have been proved;
sampling the parameter domains is not used to prove a contour bound.

\subsection{Elementary bounds}
For $x\ge0$, define
\[
 E_L(x)=\sum_{j=0}^{L}\frac{x^j}{j!},\qquad
 E_L^+(x)=E_L(x)+\frac{x^{L+1}}{(L+1)!}
                        \frac1{1-x/(L+2)}.
\]
If $L+2>x$, positive terms and the ratio of consecutive omitted terms give
$E_L(x)\le e^x\le E_L^+(x)$, with strict inequalities when $x>0$.
The programs normally use $L=50$; the exact term count is retained in the
source for each invocation. Larger arguments are evaluated either with
sufficiently many terms or with products of such bounds.

For $u\ge1$, put $v=(u-1)/(u+1)$ and define
\[
 L_L(u)=2\sum_{j=0}^{L-1}\frac{v^{2j+1}}{2j+1},\qquad
 L_L^+(u)=L_L(u)+\frac{2v^{2L+1}}{(2L+1)(1-v^2)}.
\]
Then $L_L(u)\le\log u\le L_L^+(u)$.
For $0\le x<1$, an upper bound for the dilogarithm is
\[
 D_L^+(x)=\sum_{j=1}^{L}\frac{x^j}{j^2}
             +\frac{x^{L+1}}{(L+1)^2(1-x)}.
\]
The lower bound is its finite sum without the last term.
All substitutions in these formulas are rational.
The inequalities $157/50<\pi<22/7$ can be obtained from
Machin's identity $\pi=16\arctan(1/5)-4\arctan(1/239)$ using alternating
series bounds. Comparisons involving square roots are squared when both
sides are positive. In particular,
\[
 2\pi>25/4
 \quad\Longrightarrow\quad
 \frac8{3\sqrt{2\pi}}<\frac{16}{15},
\]
which is the normalized first-moment comparison used in the intermediate
range.

For $0\le s\le6$, a cosine upper bound used in the angular sums is
\[
 C^+(s)=\sum_{j=0}^{14}\frac{(-1)^j s^{2j}}{(2j)!}.
\]
For $0\le s\le22/7$, the corresponding sine lower bound is
\[
 S^-(s)=\sum_{j=0}^{11}\frac{(-1)^j s^{2j+1}}{(2j+1)!}.
\]
The first omitted terms have the required signs, and the omitted tails
decrease in absolute value on these intervals. Thus the alternating-tail
comparison is valid even when the initial terms of the series do not decrease.
Rounding down by $10^{-k}$ means $10^{-k}\lfloor10^k x\rfloor$;
rounding up means $10^{-k}\lceil10^k x\rceil$. Both are exact rational
operations.

\subsection{Polynomial positivity}
For a rational polynomial $p$ of degree at most $D$, define
\[
 \mathcal M_{a,b,D}(p)(s)=(b+s)^D
                 p\left(\frac{a+s}{b+s}\right).
\]
This is again a polynomial. If all its coefficients are nonnegative, then
$p(X)\ge0$ for $a/b\le X<1$ when $b>a\ge0$.
The value at $1$ follows by continuity.
The transformation with $(a,b)=(0,1)$ covers $[0,1]$.
This criterion is used only as a sufficient condition; no necessity is
assumed.

Let $N_m(X)=X^m\sum_{j=0}^{4-m}X^{2j}$, so $w_m=N_m/D(X)^2$.
The polynomial $D^2-5N_1$, transformed with $(a,b,D)=(0,1,8)$, has
coefficient list
\[
 (1,5,10,10,10,20,30,20,5).
\]
It proves $w_1\le1/5$.
The derivative numerator factors as
\[
 N_1'D-2N_1D'=(1-X)q_1(X),\qquad
 q_1=1-4X^3-7X^4-7X^5-4X^6+X^9.
\]
Since $q_1'<0$ on $(0,1)$, $q_1(0)>0$, and $q_1(1)<0$,
$w_1$ has a single maximum. For $m=2,3,4$ the transformed derivative
numerators $N_m'D-2N_mD'$ have nonnegative coefficients, proving their
monotonicity.
For the lower bounds on $[1/13,1]$, the transformed polynomials
$200N_1-13D^2$ and $200N_2-D^2$, with $(a,b,D)=(1,13,8)$, have lists
\begin{align*}
 &(178912747,9987744200,5321678020,1259534360,170589970,\\[-2pt]
 &\hspace{23mm}14415800,777700,26600,475),\\
 &(13762319,1664693800,1729338740,608459320,111390890,\\[-2pt]
 &\hspace{23mm}12103000,814100,32200,575).
\end{align*}
All are positive. The stronger weight bounds on $[1/4,1]$ and $[3/5,1]$
follow from monotonicity for $m\ge2$, and endpoint comparison for the
single-maximum function $w_1$.
The same programs expand the variance polynomials in
\eqref{eq:K2}--\eqref{eq:K3} and perform the cyclotomic computations for
the constant and linear terms of $\mathcal G_3,\mathcal G_4$ and the
Bernoulli endpoint cancellations. Arithmetic in $\QQ(\zeta)$ is reduction
modulo $1+X+X^2+X^3+X^4$, using exact rational coefficients.

\subsection{Integral sums and interval coverage}
The following finite sums implement the comparisons used in the proof.
Each is a one-sided sum with an analytic monotonicity or endpoint argument.
They are not midpoint approximations to the desired inequalities.
\begin{enumerate}
\item For \eqref{eq:forty}, the forty radial intervals run from
$1/10$ to $1/2$. Evaluate the lower bound for $A$ at the right endpoint,
and the decreasing upper bound for the other term at the left endpoint.
\item For the far arc, use the fifty intervals $[j/100,(j+1)/100]$,
the lower radial bounds stated after \eqref{eq:far-log}, and the upper
value of $H$ at the left endpoint. Every difference exceeds $\gamma$.
\item The primitive-root short-angle bound uses twenty $u$ intervals.
On each interval the minimum weight occurs at an endpoint, and the maximum
cosine occurs at an endpoint of the angle interval in $[0,6]$.
Round weights down to $10^{-7}$ and use $C^+$ for the cosine.
The long-angle mass uses a hundred-cell lower sum for the decreasing
function $W(e^{-5u/2})$; its result exceeds $0.23108294$.
\item To bound $C_0$, divide each of its four integrals into forty cells.
For cell $j$ of phase $r$, enclose the angle between
\[
 \frac{157}{50}\left(\frac r5+\frac{j}{400}\right)
 \quad\text{and}\quad
 \frac{22}{7}\left(\frac r5+\frac{j+1}{400}\right).
\]
Sine is positive and concave on the enclosing interval, so the minimum
of $S^-$ at its two endpoints is a lower bound for the sine throughout.
Use this positive minimum as a denominator and round each upper summand
up to $10^{-9}$. The resulting bound is less than $1.85439429<19/10$.
\end{enumerate}
The elementary endpoint comparisons for the concave functions of $|y|$
in Section~\ref{sec:minor}, and the exponential endpoint estimates in
Section~\ref{sec:assembly}, use the same rational bounds.

\subsection{Reproduction and limits of the certificates}
From the supplementary archive root, run
\begin{verbatim}
python3 certificates/check_constants.py
python3 certificates/paper_foundations.py
python3 certificates/check_boundary_repair.py
\end{verbatim}
The first command reproduces seven groups of exact constant checks,
including the global prefactor, boundary Gaussian estimates, joint
variation, complementary-contour losses, and the four budgets.
The second checks the auxiliary arithmetic, first-moment normalization,
polynomial bounds, and supplied finite record. The third checks the boundary
prefactors $60$ and $27$ and their use in the original error budgets.
None of these commands launches the long coefficient sweep. Python optimization must be disabled, since assertions are part of
the checks; the entry points reject execution with \texttt{-O} or
\texttt{-OO}.

For traceability, the supporting modules retain their original numerical
file identifiers in the archive. These are implementation names only;
the proof in the article does not rely on a chronological research log.
The Cubic comparisons can be reproduced with
\begin{verbatim}
python3 cubic_analytic/certificate.py
python3 cubic_finite/audit_saved.py
\end{verbatim}
The first reproduces the $3582$ exact comparisons for the Cubic analytic
estimates, including the small coefficient convolutions and theta
rectangles. The second verifies the saved finite rows and $1026$
crosschecks. Complete Cubic coefficient sweeps can be rerun with
\texttt{cubic\_finite/run\_all.sh}.
Machine-readable results and a SHA-256 manifest are included.
The manifest identifies the exact supplied files and does not certify the
analytic inferences. Reproduction of the rational comparisons and review
of the continuous inequalities are separate mathematical tasks.

\section{Cubic summation and primitive-root identities}\label{app:cubic-identities}
The identities in this appendix use the notation of Part~II.
They supply the finite and infinite-tail expansions used in the local peaks.

\subsection{Shifted Euler--Maclaurin summation}
Let $0\le a\le1$, and let $f$ have four continuous derivatives on $[0,1]$.
Repeated integration by parts with periodic Bernoulli polynomials gives
\begin{align*}
 \sum_{j=0}^{n-1}f((j+a)/n)
 &=n\int_0^1 f(u)\,du
   +\sum_{k=1}^3\frac{B_k(a)}{k!n^{k-1}}
       \bigl(f^{(k-1)}(1)-f^{(k-1)}(0)\bigr)+E,\\
 |E|&\le\frac{|f'''(0)|+|f'''(1)|+
                  \int_0^1|f''''(u)|\,du}{24n^3}.
\end{align*}
The estimate follows by retaining the fourth endpoint term and bounding
both that term and the periodic fourth Bernoulli remainder by
$\sup_{0\le x\le1}|B_4(x)|\le1$.
If the expansion stops after $B_2$, the same argument with $B_3$ gives
\[
 |E|\le\frac{|f''(0)|+|f''(1)|+\int_0^1|f'''(u)|\,du}{6n^2}.
\]
On $[1,\infty)$, when $f$ and its derivatives decay absolutely, the
remainder after $B_3$ is at most
\[
 \frac{|f'''(1)|+\int_1^\infty|f''''(u)|\,du}{24n^3}.
\]
The formulas follow first on finite intervals and then by absolute
convergence. Applying them with $a=1-s/3$ to
$f_s(u)=\log(1-\omega^{-hs}e^{-3zu})$, for $s=1,2$, gives the endpoint
factor $\lambda_h$, the common correction $C/n$, and
\[
 J_h(z)=\frac{z^2}{18}\left\{
 \frac{\omega^{-h}e^{-3z}}{(1-\omega^{-h}e^{-3z})^2}
 -\frac{\omega^{-2h}e^{-3z}}{(1-\omega^{-2h}e^{-3z})^2}\right\},
 \qquad J_2=-J_1.
\]
The explicit derivative and folding estimates for the resulting remainder
are given in Part~II.

\subsection{Primitive-root eta identity}
Put $\mathscr P(w)=(w;w)_\infty$ and
$P_\infty(q)=\mathscr P(q)/\mathscr P(q^3)$. For $\Rea z>0$ and $h=1,2$,
\begin{equation}\label{cub:eq:eta-exact}
 P_\infty(\omega^he^{-z/n})
 =e^{-\pi\ii s(h,3)}e^{cn/z-z/(12n)}
 \frac{\mathscr P(\omega^{-h^{-1}}e^{-\kappa n/z})}
      {\mathscr P(e^{-3\kappa n/z})},
\end{equation}
where $h^{-1}$ is taken modulo three,
$s(1,3)=1/18$, $s(2,3)=-1/18$, $c=\pi^2/27$, and $\kappa=12c$.
To prove it, write $w=z/n$ and
$\tau=h/3+\ii w/(2\pi)$, so that
$P_\infty(q)=q^{1/12}\eta(\tau)/\eta(3\tau)$.
Choose a determinant-one matrix with lower row $(3,-h)$ and upper-left
entry $a\equiv-h^{-1}\pmod3$. Its transformed argument is
$a/3+2\pi\ii/(9w)$. Apply the eta transformation
\eqref{eq:eta-transform} to the numerator and the inversion formula to
$\eta(3\tau)$. The square roots cancel. The remaining phase is
$e^{-\pi\ii s(h,3)}$, the real exponential is
$e^{\pi^2/(27w)-w/12}$, and the two transformed product variables are
those in \eqref{cub:eq:eta-exact}.

Let $u_*=e^{-\kappa n\Rea(1/z)}$. When $u_*\le1/100$, the logarithm
of the transformed product quotient in \eqref{cub:eq:eta-exact} has
modulus less than $4u_*$. Indeed, expanding each logarithmic product gives
the bound
$u_* /(1-u_*)^2+u_*^3/(1-u_*^3)^2<4u_*$.
For $P_n=P_\infty T_n$, the exact tail logarithm is
\[
 \log T_n(\omega^he^{-z/n})
 =\sum_{\ell\ge1}\frac{e^{-3z\ell}}{\ell}
 \frac{\omega^{h\ell}e^{-z\ell/n}+\omega^{2h\ell}e^{-2z\ell/n}}
      {1-e^{-3z\ell/n}}.
\]
This follows directly from the omitted factors, and converges absolutely
for $\Rea z>0$. Combining it with the shifted tail summation formula
gives the expansions used in the intermediate and large peaks.

\end{document}